\documentclass[preprint,12pt]{elsarticle}
\usepackage{graphicx}
\usepackage{amsmath}
\usepackage{amssymb}
\usepackage[all]{xy}
\usepackage{enumerate}
\usepackage{array}
\usepackage{amsthm}
\usepackage{amscd}
\usepackage{color}
\usepackage[colorlinks=true]{hyperref}
\usepackage{endfloat}
\usepackage{array,graphicx}
\usepackage{booktabs}
\usepackage{pifont}
\def \inf{{\rm inf}}

\def \dim{{\rm dim \,}}
\def \dm{{\rm diam \,}}

\def \cd{{\rm Card \,}}

\def \cl{{\rm cl \,}}
\newcommand{\lto}{\longrightarrow}

\newcommand \Z{\mathbb{Z}}
\newcommand \N{\mathbb{N}}
\newcommand \R{\mathbb{R}}
\newcommand \F{\mathcal{F}}

\newcommand \St{{\rm St}}

\newcommand \ef{\mathbf{\Gamma}}
\newcommand \gf{\mathbf{\Delta}}

\newcommand \dih{\dim_{H}}
\newcommand \uno{\dim_{\ef}^{1}}
\newcommand \dos{\dim_{\ef}^{2}}
\newcommand \tres{\dim_{\ef}^{3}}

\newcommand \tresb{\dim_{\gf}^{3}}
\newcommand \cuatro{\dim_{\ef}^{4}}
\newcommand \cinco{\dim_{\ef}^{5}}
\newcommand \seis{\dim_{\ef}^{6}}
\newcommand \h{\dim_{H}}
\newcommand \bc{\dim_{B}}
\newcommand \dc{\dim(\alpha)}

\newcommand \cldim{\cl-\dim}

\newcommand \anf{\mathcal{A}_{n}(F)}
\newcommand \mpty{\emptyset}

\newcommand*\rot{\rotatebox{90}}
\newcommand*\OK{\ding{51}}

\newtheorem{prop}{Proposition}[section]
\newtheorem{teo}[prop]{Theorem}
\newtheorem{ct}[prop]{Counterexample}

\newtheorem{cor}[prop]{Corollary}
\newtheorem{lema}[prop]{Lemma}
\newtheorem{defn}[prop]{Definition}
\newtheorem{obs}[prop]{Remark}
\newtheorem{prob}[prop]{Open question}
\theoremstyle{definition}

\theoremstyle{remark}

\newtheoremstyle{named}{}{}{\itshape}{}{\bfseries}{.}{.5em}{\thmnote{#3's }#1}
\theoremstyle{named}
\newtheorem*{namedtheorem}{Theorem}

\journal{Topology and its Applications}

\begin{document}

\begin{frontmatter}


%% \ead{fmm124@ual.es}

%\title{Fractal Dimension for Fractal Structures: A Hausdorff Approach}


%\author{M. Fern\'andez-Mart\'{\i}nez and M.A. S\'anchez-Granero}
%\ead{fmm124@ual.es and misanche@ual.es}
%\address{Area of Geometry and Topology \\ Faculty of Science \\
%Universidad de Almer\'{\i}a \04120 Almer\'{\i}a \\ Spain}

\title{Counterexamples in theory of fractal dimension for fractal structures} 
%with applications to IFS-attractors}

\author{M. Fern\'andez-Mart\'{\i}nez\fnref{fn1}\corref{cor1}}
\ead{fmm124@gmail.com}
\address{University Centre of Defence at the Spanish Air Force Academy, MDE-UPCT,\\ 30720 Santiago de la Ribera, Murcia, SPAIN}
\fntext[fn1]{The first author specially acknowledges the valuable support
provided by Centro Universitario de la Defensa en la Academia General del
Aire de San Javier (Murcia, Spain).}

\author{Magdalena Nowak\fnref{fn2}}
\ead{magdalena.nowak805@gmail.com}
\address{Jan Kochanowski University in Kielce, \'Swietokrzyska 15, 25-406 Kielce, POLAND}
\fntext[fn2]{The second author was partially supported by National Science Centre grant DEC-2012/07/N/ST1/03551.}

\author{M.A S\'anchez-Granero\fnref{fn3}}
\ead{misanche@ual.es}
\address{Department of Mathematics, Universidad de Almer\'{\i}a, 04120 Almer\'{\i}a, SPAIN}
\fntext[fn3]{The third author acknowledges the support of the Ministry of Economy and Competitiveness
of Spain, Grant MTM2012-37894-C02-01.}

\cortext[cor1]{Corresponding author}
%\cortext[cor2]{Principal corresponding author}
%\fntext[fn2]{The first author acknowledges a FPU grant of the Spanish Ministry of Education.}
%\fntext[fn3]{The second author acknowledges the support of the Ministry of Economy and Competitiveness
%of Spain, Grant MTM2012-37894-C02-01.}

%\address[rvt]{University Centre of Defence at the Spanish Air Force Academy, MDE-UPCT\\ Coronel L\'opez Pe\~{n}a Street, w/n\\
%30720 Santiago de la Ribera, Murcia, SPAIN}

%\address[rvt1]{Area of Geometry and Topology\\ Department of Mathematics\\
%Universidad de Almer\'{\i}a\\ 04120 Almer\'{\i}a (Spain)}

\begin{abstract}
Fractal dimension constitutes the main tool to test for fractal patterns in Euclidean contexts. For this purpose, it is always used the box dimension, since it is easy to calculate, though the Hausdorff dimension, which is the oldest and also the most accurate fractal dimension, presents the best analytical properties. Additionally, fractal structures provide an appropriate topological context where new models of fractal dimension for a fractal structure could be developed in order to generalize the classical models of fractal dimension. In this paper, we provide some counterexamples regarding these new models of fractal dimension in order to show the reader how they behave mathematically with respect to the classical models, and also to point out which features of such models can be exploited to powerful effect in applications.
\end{abstract}


\end{frontmatter}

\section{Introduction}

The analysis of fractal patterns have growth increasingly during the last years, mainly due to the wide range of applications to diverse scientific areas where fractals have been explored, including physics, statistics, and economics (see, e.g., \cite{FAL90,FED88}). It is also worth mentioning that there has also been a special interest for applying fractals to social sciences (see for example \cite{BRO10B} and references therein). 

It turns out that the key tool to study the complexity of a given system is the fractal dimension, since this is its main invariant which throws quite useful information about the complexity that it presents when being examined with enough level of detail. 
%Likewise, the fractal dimension theory has been applied in some fields of science, such as the study of dynamical systems \cite{FAN10}, diagnosis of diseases, such as osteoporosis \cite{RUT92} or cancer \cite{BAI00}, ecology \cite{BAR03}, earthquakes \cite{HIR89}, detection of eyes in human face images \cite{LIN01}, and the analysis of the human retina \cite{LAN03}, just to name a few.

%Fractal dimension is a single quantity which throws some useful information about the irregularities that a given space presents when being examined through a given range of scales. According to 
We would like also to point out that fractal dimension is usually understood as the classical box dimension, mainly in the field of empirical applications. In fact, its popularity is due to the possibility of its effective calculation and empirical estimation. On the other hand, the Hausdorff dimension also constitutes a powerful analytical model which allows to ``measure'' the complexity of a system, at least from a theoretical point of view. Nevertheless, though they are defined for any metric (resp. metrizable) space, almost all the empirical applications of fractal dimension are tackled in the context of Euclidean spaces. 
In addition to that, recall that box dimension is more useful for practical applications, whereas Hausdorff dimension presents ``better'' analytical properties, due to the fact that its standard definition is based on a measure. Indeed, though Hausdorff dimension becomes the most accurate model for fractal dimension, since its definition is quite general, it can result difficult or even impossible to calculate in practical applications. 

It is worth mentioning that the application of fractal structures allows to provide new models for a fractal dimension definition on any generalized-fractal space, and not only on the Euclidean ones. This extends the classical theory of fractal dimension to the more general context of fractal structures. In this way, some theoretical results have been shown to generalize the classical models in the context of fractal structures (see, e.g., \cite[Theorem 4.15]{DIM3} and \cite[Theorem 3.12]{DIM4}). Moreover, we would like to point out that some fractal dimensions for a fractal structure have been already succesfully applied in non-Euclidean contexts, where the box dimension cannot be applied (see, e.g., \cite{MFM12,MFM13}).
%Recall that we have developed different kind of ideas to provide new definitions of fractal dimension for fractal structures. They can be classified into the next two different points of view.

Accordingly, when providing a new model to calculate the fractal dimension, it would be desirable that the contributed definitions allow to calculate the fractal dimension for a given subset as easy as the box dimension models, though one should be also mirrored in the analytical properties satisfied by the Hausdorff dimension.

In this paper, though, we provide some counterexamples regarding some new models of fractal dimension (described in forthcoming Subsection \ref{sub:models}) in order to show the reader how they behave mathematically with respect to the classical models, and also to point out which features of such models can be exploited to powerful effect in further applications.

The organization of this paper is as follows. In Section \ref{sec:pre}, we provide some basic definitions, notations, and results that are useful along this paper.
In Section \ref{sec:def}, we provide a collection of counterexamples regarding some theoretical properties that are often explored for new definitions of fractal dimension. 
Moreover, in Section \ref{sec:dim1}, we describe some features about a fractal dimension model for a fractal structure which does not depend on any metric.
In addition to that, we devote Section \ref{sec:defs} to justify some natural assumptions that must be satisfied in some Hausdorff dimension type definitions for a fractal structure.
It is also worth mentioning that some counterexamples regarding the fractal dimension of IFS-attractors are contained in Section \ref{sec:ifs}.
Finally, in Section \ref{sec:agree}, we establish that some of the fractal dimension definitions for a fractal structure that we analyze along this paper, do not coincide, in general.

\section{Preliminaries}\label{sec:pre}

In this section, we provide all the necessary mathematical background for this paper, including definitions, notions, and theoretical results.\newline 
Along this paper, let $I=\{1,\ldots,k\}$ be a finite set of indices.
%{\color{red}According to that, the structure of this preliminary section is as follows. 
%The next two subsections do contain the basic definitions in order to calculate the fractal dimension for a given space. 
%In fact, in Subsection \ref{sub:classics}, we recall the classical models for fractal dimension (namely, both the box and the Hausdorff dimensions), whereas in upcoming Subsection \ref{sub:models}, some new models of fractal dimension specially designed with respect to any fractal structure, are described. 
%
%In Subsection \ref{sub:fs}, the concept of fractal structure is presented from a topological point of view. It is worth mentioning that, also in such subsection, the notion of a natural fractal structure which any Euclidean subspace can be equipped with, is introduced.
%
%Moreover, Subsection \ref{sub:ifs} includes the classical notion of an IFS-attractor, the natural fractal structure which any IFS-attractor can be equipped with, and also recalls the OSC, as well. Finally, to end this section, we devote Subsection \ref{sub:moran} to state the classical Moran's Theorem (1946).}

\subsection{Classical models for fractal dimension}\label{sub:classics}

%{\color{red}Fractal dimension consists of a single quantity which yields valuable information about the complexity that a given space presents, provided that it is explored with enough level of detail. 
%Thus, fractal dimension is usually understood as the classical box-counting dimension, which is also known as information dimension, Kolmogorov entropy, capacity dimension, entropy dimension, metric dimension, logarithmic density, \ldots, etc (see \cite{FAL90A}).
%Though the Hausdorff dimension can be considered a fractal dimension too, in practical applications it is always used the box-counting dimension, since it is the only one that can be calculated when working with a finite range of scales, which is the case of empirical applications. Popularity of the box-counting dimension is mainly due to the possibility of its effective calculation and empirical estimation in Euclidean contexts. Indeed, in practical applications, the box-counting dimension can be estimated as the slope of the regression line of a $\log$-$\log$ graph plotted for a suitable discrete collection of scales.
%The basic theory of box-counting dimension can be found in \cite{FAL90A}. 
Next, we recall the definition of the standard box dimension.
%{\color{red}, which is mainly used in empirical applications of fractal dimension, due to the easiness in its empirical estimation or even explicit calculation, as well. }
As \cite[Subsection 3.6]{FAL90} points out, its origins become quite hard to trace, though it seems that it would have been considered previously by the Hausdorff dimension pioneers, who rejected it at a first glance due to its lack of theoretical properties. Anyway, the standard definition of box dimension that we recall next, was firstly provided in \cite{PON32}.

\begin{defn}\label{def:bc}
The (lower/upper) box dimension for any subset $F\subseteq \mathbb{R}^d$ is given by the following (lower/upper) limit:
\begin{equation*}\label{eq:bc}
\bc(F)=\lim_{\delta\to 0}\frac{\log N_{\delta}(F)}{-\log \delta},
\end{equation*}
%where $\delta$ is the scale, and $N_{\delta}(F)$ is
%can be calculated equivalently through any of the following quantities:
% (see \cite[Equivalent Definitions 3.1]{FAL90A}):
%\begin{enumerate}[(1)]
%\item the number of $\delta$-cubes that meet $F$, where a $\delta$-cube in $\mathbb{R}^d$ is a set of the form \newline $[k_{1}\delta,(k_{1}+1)\delta]\times [k_{2}\delta,(k_{2}+1)\delta]\times \ldots \times [k_{d}\delta,(k_{d}+1)\delta]$ where $k_i\in \mathbb{Z}$ for all $i\in \{1,\ldots,d\}$.
%\item 
where $N_{\delta}(F)$ is the number of $\delta$-cubes that intersect $F$.\label{eq:bc3}
%\item the smallest number of sets of diameter at most $\delta$ that cover $F$.\label{eq:bc4}
%the largest number of disjoint balls of radii $\delta$ having centres in $F$.\label{eq:bc5}
%\end{enumerate}
\end{defn}
Recall that a $\delta$-\emph{cube} in $\mathbb{R}^d$ is a set of the form $[k_{1}\, \delta,(k_{1}+1)\, \delta]\times\, \ldots\, \times [k_{d}\, \delta,(k_{d}+1)\, \delta]$, with $k_1,\,\ldots\,,k_d\in \mathbb{Z}$. It is also worth mentioning that the formula contained in Definition \ref{def:bc} could be properly discretized through $\delta=1/2^n:n\in \N$, which becomes specially appropriate for computational purposes \cite{MFM12}.
Some alternatives to calculate the box dimension could be found out in \cite[Equivalent definitions 3.1]{FAL90}, where equivalent expressions to calculate $N_{\delta}(F)$ are provided. Moreover, notice that in both \cite[Theorem 3.5]{DIM1} and \cite[Equivalent definitions 2.1]{FAL14}, the equivalence among all of these alternative approaches to calculate the box dimension are shown. 
%Moreover, \cite[Theorem 3.5]{DIM1} completes
%Note also that the limit in Eq. (\ref{eq:bc}) can be discretized by taking, for example, $\delta=1/2^{n}$, which is formalized in the next remark.
%\begin{obs}\label{obs:0}
%To calculate the (lower/upper) box-counting dimension of any subset $F$ of an Euclidean space $\mathbb{R}^d$, it suffices with taking limits as $\delta\to 0$ through any decreasing sequence $\{\delta_n\}_{n\in \mathbb{N}}$ verifying that $\delta_{n+1}\geq c\cdot \delta_n$ for all $n\in \mathbb{N}$, where $c\in (0,1)$ is a suitable constant. In particular, it holds for $\delta_n=1/2^{n}$.
%\end{obs}
%The first to define a measure by means of coverings of sets was Carath\'eodory in \cite{CAR14}. 

On the other hand, in 1919, Hausdorff applied a methodology developed by Carath\'eo\-dory some years earlier (see \cite{CAR14}) in order to define the measures that now bear his name, and showed that the middle third Cantor set has positive and finite measure whose dimension is equal to $\log 2/ \log 3$ \cite{HAU19}. In addition to that, a detailed study regarding the analytical properties of both the Hausdorff measure and dimension was mainly developed by Besicovitch and his pupils afterwards (see, e.g., \cite{BES34,BES37}).

%The Hausdorff dimension, which is the oldest definition of fractal dimension, presents the best analytical properties. Indeed, note that this fractal dimension can be defined for any subset of an Euclidean (resp. metrizable) space and its definition is based on a measure which makes it very convenient from a mathematical point of view. Nevertheless, it presents some disadvantages, particularly from the point of view of applications, since it can be hard to calculate or to estimate.
%Therefore, while this fractal dimension is ``better" from a theoretical approach, the box-counting dimension is ``better" for a wide range of applications. 

Along this paper, we will define the diameter of a given subset $A$ of any metric space $(X,\rho)$, as usual, by $\dm(A)=\sup\{\rho(x,y):x,y\in A)\}$. 
Next, let us recall the standard construction regarding the Hausdorff dimension. 
%Thus, let $(X,\rho)$ be a metric space, and let $\delta$ be a positive real number. 
Let $\delta>0$.
Thus, for any subset $F$ of $X$, we recall that a $\delta$-\emph{cover} of $F$ is just a countable family of subsets $\{U_j\}_{j\in J}$, such that $F\subseteq \bigcup_{j\in J}U_j$, where $\dm(U_j)\leq \delta$, for all $j\in J$. Moreover, let $\mathcal{C_\delta}(F)$ be the collection of all $\delta$-covers of $F$, and 
%The underlying idea to define the Hausdorff measure consists of minimizing the sum of the $s$-powers of the diameters of all the subsets for any $\delta$-cover, where $s$ is going to be the fractal dimension. 
let us consider the following quantity:
\begin{equation*}\label{eq:hdelta}
\mathcal{H}_\delta^s(F)=\inf\Bigg\{\sum_{j\in J}\dm(U_j)^s:\{U_j\}_{j\in J}\in \mathcal{C}_\delta(F)\Bigg\}.
\end{equation*}
%Note that when $\delta$ decreases, then the class $\mathcal{C}_\delta(F)$ of all $\delta$-covers of $F$ is reduced, so the measure of $F$ increases. Accordingly, the next limit always exists:
Interestingly, the next limit always exists:
\begin{equation*}\label{eq:hm}
\mathcal{H}_H^s(F)=\lim_{\delta\to 0}\mathcal{H}_\delta^s(F),
\end{equation*}
which is named as the $s$-\emph{dimensional Hausdorff measure} of $F$. 
%Hausdorff measure generalizes the classical Lebesgue measure for Euclidean subspaces.
Thus, the \emph{Hausdorff dimension} of $F$ is fully determined as the unique point $s$, where the $s$-dimensional Hausdorff measure ``jumps'' from $\infty$ to $0$, namely,
% is called the Hausdorff dimension of $F$ (also called the Hausdorff-Besicovitch dimension) as can be seen on Figure \ref{fig:haus}.
%Indeed, note that the Hausdorff dimension may be described as follows:
\begin{equation}\label{eq:hreach}
\dih(F)=\inf\{s:\mathcal{H}_H^s(F)=0\}=\sup\{s:\mathcal{H}_H^s(F)=\infty\}.
\end{equation}

\subsection{Fractal structures and the natural fractal structure on any Euclidean subspace}\label{sub:fs}

The concept of fractal structure, which naturally appears in several topics regarding Asymmetric Topology \cite{SG10}, was first introduced in \cite{SG99A} to characterize non-Archimedeanly quasi-metrizable spaces. Afterwards, in \cite{MSG12}, it was applied to deal with IFS-attractors.
It is worth mentioning that fractal structures do constitute a powerful tool to develop new fractal dimension models that allow to calculate the fractal dimension over a wide range of (non-Euclidean) spaces and contexts (see, e.g., \cite{MFM12}).

Recall that a family $\Gamma$ of subsets of a given space $X$ is said to be a \emph{covering}, if $X=\bigcup\{A:A\in \Gamma\}$. 
%Thus, a fractal structure is a countable collection of coverings which provides better approximations to the whole space as deeper stages are reached, which we will refer to as levels of such a fractal structure.

%According to that, whether $\Gamma$ is a covering of $X$, we will denote $\St(x,\Gamma)=\bigcup\{A\in \Gamma: x\in A\}$, and $U_{x\, \Gamma}=X\setminus\bigcup\{A\in \Gamma: x\notin A\}$. Further, if $\ef=\{\Gamma_n\}_{n\in \mathbb{N}}$ is a countable family of coverings of $X$, then we will denote $U_{xn}=U_{x\Gamma_n}$, $\uxg=\{U_{xn}\}_{n\in \mathbb{N}}$, and $\St(x,\ef)=\{\St(x,\Gamma_n)\}_{n\in \mathbb{N}}$, as well.

Next, we provide the definition of a fractal structure on a set $X$. In fact, let $\Gamma_1$ and $\Gamma_2$ be any two coverings of $X$. Recall that $\Gamma_1\prec \Gamma_2$ means that $\Gamma_1$ is a \emph{refinement} of $\Gamma_2$, namely, for all $A\in \Gamma_1$, there exists $B\in \Gamma_2$, such that $A\subseteq B$. In addition to that, $\Gamma_1\prec \prec \Gamma_2$ denotes that $\Gamma_1\prec \Gamma_2$, and also that, for all $B\in \Gamma_2$, $B=\bigcup\{A\in \Gamma_1: A\subseteq B\}$. Hence, a fractal structure on a set $X$ is a countable family of coverings of $X$, $\ef=\{\Gamma_n\}_{n\in \mathbb{N}}$, such that $\Gamma_{n+1}\prec \prec \Gamma_n$, for all $n\in \mathbb{N}$. In this way, covering $\Gamma_n$ is called \emph{level} $n$ of the fractal structure $\ef$. 

A fractal structure induces a transitive base of quasi-uniformity (and hence a topology) given by the transitive family of entourages $U_{\Gamma_n}=X\setminus\bigcup\{A\in \Gamma_n: x\notin A\}$, where $n \in \mathbb{N}$. 
To simplify the theory, the levels in a fractal structure will not be coverings in the usual sense. Instead of this, we allow that a set can appear twice or more in any level of a fractal structure. Also, we would like to point out that a fractal structure $\ef$ is said to be \emph{finite} provided that all its levels $\Gamma_n$ are finite coverings. Further, a fractal structure $\ef$ is said to be \emph{locally finite}, if for each level $n$ in that fractal structure, it holds that any point $x\in X$ belongs to a finite number of elements $A\in \Gamma_n$. 
In general, if $\Gamma_n$ satisfies a certain property $P$, for all $n\in \mathbb{N}$, and $\ef=\{\Gamma_n\}_{n\in \mathbb{N}}$ is a fractal structure on $X$, then we will say that $\ef$ is a fractal structure within the property $P$, and also that $(X,\ef)$ is a GF-space with that property.
%For instance, $\Gamma_1=\{[0,1/2],[1/2,1],[0,1/2]\}$ may be the first level of a fractal structure defined on the closed unit interval $[0,1]$.
%\end{obs}
%Recall also that if $\ef$ is a pre-fractal structure, then any of its levels is a closure-preserving closed covering (see \cite[Proposition 2.4]{SG02B}).
In addition, if $\ef$ is a fractal structure on $X$, and $\St(x,\ef)=\{\St(x,\Gamma_n):n\in \mathbb{N}\}$ is a neighborhood base of $x$ for each $x\in X$, where $\St(x,\Gamma_n)=\bigcup\{A\in \Gamma_n: x\in A\}$, then $\ef$ is said to be a \emph{starbase fractal structure}. It is worth mentioning that starbase fractal structures are connected with metrizability (see \cite{SG02A,SG02B}).
%A fractal structure $\ef$ is said to be finite, if all levels $\Gamma_n$ are finite coverings. 
%Moreover, a fractal structure $\ef$ is said to be $\ef$-Cantor-complete if for each decreasing sequence $\{A_n\}_{n\in \mathbb{N}}$ (namely, $A_{n+1}\subseteq A_n$ for all $n\in \mathbb{N}$) of subsets of $X$ with $A_n \in \Gamma_n$, then it holds that $\bigcap_{n \in \N} A_n \not=\emptyset$.

On the other hand, it turns out that any Euclidean space $\R^d$ can always be equipped with a natural fractal structure, which satisfies some interesting topological properties. In fact, such a natural fractal structure, which was first described in \cite[Definition 3.1]{DIM1}, is locally finite, starbase and induces the usual topology.
%was first sketched in \cite{BA92}, and formally defined later in \cite[Definition 4.4]{MSG12}. 
%Thus, while the latter provides the definition of a fractal structure as a mathematical concept, the former ``talks'' about the (natural) fractal structure of a self-similar set. 
%Next, we recall the description of such a fractal structure, which becomes essential for upcoming sections.
\begin{defn}\label{def:1}
The natural fractal structure on a Euclidean space $\mathbb{R}^{d}$ is given by the countable family of coverings $\ef=\{\Gamma_{n}\}_{n\in \mathbb{N}}$, whose levels are 
\begin{equation*}
\Gamma_{n}=\Bigg\{\Bigg[\frac{k_{1}}{2^{n}},\frac{k_{1}+1}{2^{n}}\Bigg]\times \, \ldots\, \times \Bigg[\frac{k_{d}}{2^{n}},\frac{k_{d}+1}{2^{n}}\Bigg]:k_{1},\ldots,k_{d}\in \mathbb{Z}\Bigg\}.
\end{equation*}
%for each natural number $n$.
\end{defn}

%\begin{obs}\label{obs:induced}
In particular, a natural fractal structure induced on real subsets could also be considered from Definition \ref{def:1}.
%In particular, it is also possible to consider a natural fractal structure induced on real subsets from Definition \ref{def:1}. 
For instance, the natural fractal structure (on the real line) induced on the closed unit interval $[0,1]$, could be defined as the countable family of coverings $\ef=\{\Gamma_n\}_{n\in \mathbb{N}}$, whose levels are given by $\Gamma_n=\{[\frac{k}{2^n},\frac{k+1}{2^n}]:k\in \{0,1,\ldots, 2^{n}-1\}\}$.
%\end{obs}
%
\subsection{Theoretical models for fractal dimension based on fractal structures}\label{sub:models}

Next, we recall our two first models of fractal dimension for a fractal structure, namely, both fractal dimensions I \& II. Thus, as it happens with classical box dimension, these models of fractal dimension do not have always to exist. This is the reason for which we have to define them through lower/upper limits, too.

\begin{defn}[Box dimension type models for a fractal structure]\label{def:typebox}
Let $\ef$ be a fractal structure on a distance space $(X,\rho)$, $F$ be a subset of $X$, and $N_{n}(F)$ be the number of elements in level $n$ that intersect $F$. Thus,
\begin{enumerate}[(1)]
\item the (lower/upper) fractal dimension I for $F$ is given by the (lower/upper) limit: \label{def:dim1}
\begin{equation*}
\uno(F)=\lim_{n\to \infty}\frac{\log N_{n}(F)}{n\log 2}.
\end{equation*}
\item The (lower/upper) fractal dimension II for $F$ is defined as the (lower/upper) limit: \label{def:dim2}
\begin{equation*}
\dos(F)=\lim_{n\to \infty}\frac{\log N_{n}(F)}{-\log \delta(F,\Gamma_{n})},
\end{equation*}
where $\delta(F,\Gamma_{n})=\sup\{\dm(A):A\in \anf\}$ is the diameter of $F$ in each level of the fractal structure, and $\anf=\{A\in \Gamma_n:A\cap F\neq \mpty\}$.
%as provided in Definition \ref{def:3}.
\end{enumerate}
\end{defn}

It turns out that both fractal dimensions I \& II do generalize the box dimension in the context of Euclidean spaces equipped with their natural fractal structures (see \cite[Theorem 4.7]{DIM1}).

It is worth mentioning that the fractal dimension definition for a fractal structure we provide next, could be understood as a hybrid model, since its definition is made, someway, as a discrete version regarding the Hausdorff dimension, though it also generalizes the classical box dimension, too.
%(see \cite[Theorem 4.15]{DIM3}).
\begin{defn}\label{def:dim3}
%[Hausdorff dimension type models for a fractal structure]\label{}
Let $\ef$ be a fractal structure on a metric space $(X,\rho)$, $F$ be a subset of $X$, and assume that $\delta(F,\Gamma_n)\to 0$.
%, and let us consider the following expression:
%\begin{enumerate}[(1)]
%\item 
Given $n\in \N$, let us also consider the next expression: \label{eq:hnk}
$$\mathcal{H}_{n,3}^{s}(F)=\inf\Bigg\{\sum_{j\in J}\dm(A_j)^s:\{A_j\}_{j\in J}\in \mathcal{A}_{n,3}(F)\Bigg\},$$
where $\mathcal{A}_{n,3}(F)=\{\{A \in \Gamma_l: A \cap F \neq\mpty\}: l\geq n \}$. 
%is given, in each case, as follows:
%$$\mathcal{A}_{n,k}(F)= \left\{ \begin{array}{ll}
%         \Big\{\{A_{i}\}_{i\in I}:A_{i}\in \bigcup_{l\geq n}\Gamma_{l}\ \textrm{for all} \ i\in I, F\subseteq \bigcup_{i\in I}A_{i}, \cd(I)<\infty\Big\} & \mbox{if $k=4$}\\
%         \Big\{\{A_{i}\}_{i\in I}:A_{i}\in \bigcup_{l\geq n}\Gamma_{l} \ \textrm{for all} \ i\in I, F\subseteq \bigcup_{i\in I}A_{i} \Big\} & \mbox{if $k=5$}
%        \end{array} \right.$$
%\begin{enumerate}[(i)]
%\item $ \{\{A \in \Gamma_l: A \cap F \not=\emptyset \}: l \geq n \}$, \hfill if $k=3$; %\label{eq:hn3}
%\item $\{\{A_{j}\}_{j\in J}:A_{j}\in \bigcup_{l\geq n}\Gamma_{l},\forall\, j\in J, F\subseteq \bigcup_{j\in J}A_{j}, \cd(J)<\infty\}$,\hfill if $k=4$;
%\item $\{\{A_{j}\}_{j\in J}:A_{j}\in \bigcup_{l\geq n}\Gamma_{l}, \forall\, j\in J, F\subseteq \bigcup_{j\in J}A_{j} \}$,\hfill if $k=5$.
%\end{enumerate}        
%for all $n\in \mathbb{N}$. 
Define also
$$\mathcal{H}_3^s(F)=\lim_{n\to \infty}\mathcal{H}_{n,3}^{s}(F).$$
%for $k\in \{3,4,5\}$, 
Thus, the fractal dimension III for $F$ is defined as the following non-negative real number:
$$\dim_{\ef}^{3}(F)=\inf\{s:\mathcal{H}_3^s(F)=0\}=\sup\{s:\mathcal{H}_3^s(F)=\infty\}.$$
%for $k\in \{4,5\}$.
%\item Consider the following expression:
%$$\mathcal{H}_{\delta,6}^{s}(F)=\inf\Bigg\{\sum_{j\in J}\dm(A_j)^s:\{A_j\}_{j\in J}\in \mathcal{A}_{\delta,6}(F)\Bigg\},$$
%where for each $\delta>0$, the family $\mathcal{A}_{\delta,6}(F)$ is defined by
%%$$\mathcal{A}_{n,k}(F)= \left\{ \begin{array}{ll}
%%         \Big\{\{A_{i}\}_{i\in I}:A_{i}\in \bigcup_{l\geq n}\Gamma_{l}\ \textrm{for all} \ i\in I, F\subseteq \bigcup_{i\in I}A_{i}, \cd(I)<\infty\Big\} & \mbox{if $k=4$}\\
%%         \Big\{\{A_{i}\}_{i\in I}:A_{i}\in \bigcup_{l\geq n}\Gamma_{l} \ \textrm{for all} \ i\in I, F\subseteq \bigcup_{i\in I}A_{i} \Big\} & \mbox{if $k=5$}
%%        \end{array} \right.$$
%$$\mathcal{A}_{\delta,6}(F)=\Bigg\{\{A_j\}_{j\in J}:A_j\in \bigcup_{l\in \mathbb{N}}\Gamma_l\ \forall \, j\in J, \dm(A_j)\leq \delta, F\subseteq \bigcup_{j\in J}A_j\Bigg\}.$$     
%%for all $n\in \mathbb{N}$. 
%Thus, provided that
%$$\mathcal{H}_6^s(F)=\lim_{\delta\to 0}\mathcal{H}_{\delta,6}^{s}(F),$$
%then the fractal dimension VI of $F$ is defined as
%$$\seis(F)=\inf\{s:\mathcal{H}_6^s(F)=0\}=\sup\{s:\mathcal{H}_6^s(F)=\infty\}.$$
%for $k\in \{4,5\}$.
%\end{enumerate}
%\item 
\end{defn}
Unlike it happens with fractal dimensions I \& II (recall Definition \ref{def:typebox}), fractal dimension III always exist (see \cite[Remark 4.4]{DIM3}). This is mainly due to the fact that $\mathcal{H}_{n,3}^s$ is a monotonic sequence in each natural number $n$. Moreover, \cite[Theorem 4.15]{DIM3} shows that fractal dimension III generalizes box dimension (as well as fractal dimensions I \& II) in the context of Euclidean spaces equipped with their natural fractal structures.
Interestingly, it also turns out that fractal dimension III can be estimated as easy as the standard box dimension in empirical applications. For additional details, we refer the reader to \cite{DIM3}, and specially to Theorem 4.7, therein, since it provides a handier expression to calculate the fractal dimension III through the families $\anf$. In fact, to deal with, the following quantity allows us to calculate the fractal dimension III for a given subset $F$ of $X$:
$$\mathcal{H}_n^s(F)=\sum\{\dm(A)^s:A\in \anf\}.$$
% An algorithm to calculate fractal dimension III is given in \cite{}
%Thus, fractal dimension III becomes an intermediate model between classical Hausdorff and box dimensions.
The last step in this subsection is to describe three more models for a fractal dimension definition (with respect to any fractal structure) following the spirit of the Hausdorff dimension. Thus, while the first one is specially interesting, since its description is made in terms of finite coverings (which allowed the authors in \cite{HAU15} to contribute the first-known overall algorithm to calculate the Hausdorff dimension), the remaining definitions become close approaches to the classical Hausdorff dimension. In fact, while upcoming fractal dimension V is a discrete version regarding the Hausdorff model, the fractal dimension VI provides a continuous approach for that in terms of $\delta$-covers.
\begin{defn}[Hausdorff dimension type models for a fractal structure]\label{def:456}
Let $\ef$ be a fractal structure on a metric space $(X,\rho)$, $F$ be a subset of $X$, and assume that $\delta(F,\Gamma_n)\to 0$. Thus, let us consider the following expressions:
\begin{enumerate}[(1)]
\item given $n\in \N$: \label{eq:hnk}
$$\mathcal{H}_{n,k}^{s}(F)=\inf\Bigg\{\sum_{j\in J}\dm(A_j)^s:\{A_j\}_{j\in J}\in \mathcal{A}_{n,k}(F)\Bigg\},$$
where $\mathcal{A}_{n,k}(F)$ is given, in each case, by:
%$$\mathcal{A}_{n,k}(F)= \left\{ \begin{array}{ll}
%         \Big\{\{A_{i}\}_{i\in I}:A_{i}\in \bigcup_{l\geq n}\Gamma_{l}\ \textrm{for all} \ i\in I, F\subseteq \bigcup_{i\in I}A_{i}, \cd(I)<\infty\Big\} & \mbox{if $k=4$}\\
%         \Big\{\{A_{i}\}_{i\in I}:A_{i}\in \bigcup_{l\geq n}\Gamma_{l} \ \textrm{for all} \ i\in I, F\subseteq \bigcup_{i\in I}A_{i} \Big\} & \mbox{if $k=5$}
%        \end{array} \right.$$
\begin{enumerate}[(i)]
%\item $ \{\{A \in \Gamma_l: A \cap F \not=\emptyset \}: l \geq n \}$, \hfill if $k=3$; %\label{eq:hn3}
\item $\{\{A_{j}\}_{j\in J}:A_{j}\in \bigcup_{l\geq n}\Gamma_{l},\forall\, j\in J, F\subseteq \bigcup_{j\in J}A_{j}, \cd(J)<\infty\}$,\hfill if $k=4$;
\item $\{\{A_{j}\}_{j\in J}:A_{j}\in \bigcup_{l\geq n}\Gamma_{l}, \forall\, j\in J, F\subseteq \bigcup_{j\in J}A_{j} \}$,\hfill if $k=5$. 
\end{enumerate}        
%for all $n\in \mathbb{N}$. 
Define also:
$$\mathcal{H}_k^s(F)=\lim_{n\to \infty}\mathcal{H}_{n,k}^{s}(F),$$
for $k=4,5$. Thus, the fractal dimension IV (resp. V) for $F$ is defined as the non-negative real value satisfying the following identity:
$$\dim_{\ef}^{k}(F)=\inf\{s:\mathcal{H}_k^s(F)=0\}=\sup\{s:\mathcal{H}_k^s(F)=\infty\}.$$
%for $k\in \{4,5\}$.
\item Given $\delta>0$:
$$\mathcal{H}_{\delta,6}^{s}(F)=\inf\Bigg\{\sum_{j\in J}\dm(A_j)^s:\{A_j\}_{j\in J}\in \mathcal{A}_{\delta,6}(F)\Bigg\},$$
%where for each $\delta>0$, 
where $\mathcal{A}_{\delta,6}(F)$ is defined as
%$$\mathcal{A}_{n,k}(F)= \left\{ \begin{array}{ll}
%         \Big\{\{A_{i}\}_{i\in I}:A_{i}\in \bigcup_{l\geq n}\Gamma_{l}\ \textrm{for all} \ i\in I, F\subseteq \bigcup_{i\in I}A_{i}, \cd(I)<\infty\Big\} & \mbox{if $k=4$}\\
%         \Big\{\{A_{i}\}_{i\in I}:A_{i}\in \bigcup_{l\geq n}\Gamma_{l} \ \textrm{for all} \ i\in I, F\subseteq \bigcup_{i\in I}A_{i} \Big\} & \mbox{if $k=5$}
%        \end{array} \right.$$
$$\mathcal{A}_{\delta,6}(F)=\Bigg\{\{A_j\}_{j\in J}:A_j\in \bigcup_{l\in \mathbb{N}}\Gamma_l, \forall \, j\in J, \dm(A_j)\leq \delta, F\subseteq \bigcup_{j\in J}A_j\Bigg\}.$$     
%for all $n\in \mathbb{N}$. 
Let us consider also
$$\mathcal{H}_6^s(F)=\lim_{\delta\to 0}\mathcal{H}_{\delta,6}^{s}(F).$$
Hence, the fractal dimension VI for $F$ is given by
$$\seis(F)=\inf\{s:\mathcal{H}_6^s(F)=0\}=\sup\{s:\mathcal{H}_6^s(F)=\infty\}.$$
%for $k\in \{4,5\}$.
\end{enumerate}
\end{defn}
In Definition \ref{def:456}, we consider that $\inf\ \emptyset=\infty$. Thus, whether $\mathcal{A}_{n,4}(F)=\emptyset$ (resp. $\mathcal{A}_{n,5}(F)=\emptyset$ or $\mathcal{A}_{\delta,6}(F)=\emptyset$) for some $F\subset X$, then $\cuatro(F)=\infty$ (resp. $\cinco(F)=\infty$ or $\dim_{\ef}^6(F)=\infty$).
It is worth mentioning that Table \ref{table:1} in Section \ref{sec:def} provides a schematic comparison regarding the analytical properties that are satisfied by all the fractal dimensions studied along this paper.
Surprisingly, it turns out that fractal dimension IV would not seem to be interesting enough in the light of the properties it satisfies, since they are the same as box dimension does. However, though its definition is made in terms of finite coverings, it holds that both fractal dimension IV and Hausdorff dimension coincide for compact Euclidean subspaces (see \cite[Theorem 3.13 \& Corollary 3.14 (2)]{DIM4}). This theoretical fact allowed the authors therein to provide the first-known procedure to calculate the Hausdorff dimension in practical applications (see \cite[Algorithm 3.1]{HAU15}). On the other hand, fractal dimension IV also becomes an intermediate model between the box and the Hausdorff dimensions (for additional details, we refer the reader to \cite[Remark 3.15]{DIM4}). On the other hand, fractal dimensions V \& VI do also generalize the Hausdorff model in the context of Euclidean subspaces equipped with their natural fractal structures (see \cite[Corollary 3.11]{DIM4}).

\subsection{IFS-attractors. The natural fractal structure on IFS-attractors. The open set condition}\label{sub:ifs}

First, let $f:X\lto X$ be a self-map defined on a metric space $(X,\rho)$. Recall that $f$ is said to be a Lipschitz self-map, whenever it satisfies that $\rho(f(x),f(y))\leq c\, \rho(x,y)$, for all $x,y\in X$, where $c>0$ is the Lipschitz constant associated with $f$. In particular, if $c<1$, then $f$ is said to be a \emph{contraction}, and we will refer to $c$ as its \emph{contraction factor}.
Further, if the equality in the previous expression is reached, namely, $\rho(f(x),f(y))=c\, \rho(x,y)$, for all $x,y\in X$, then $f$ is called a \emph{similarity}, and its Lipschitz constant is its \emph{similarity factor}, too.

\begin{defn}
For a metric space $(X,\rho)$, let us define an IFS as a finite family $\F=\{f_i\}_{i\in I}$, where $f_i$ is a contraction, for all $i\in I$. Thus, the unique compact set $A\subset X$, which satisfies that $A=\bigcup_{f\in\F}f(A)$, is called the attractor of the IFS $\F$, or IFS-attractor, as well. Further, it is also called a self-similar set. If, in addition to that, all the $f_i$ mappings are similarities, then we will say that $A$ is a strict self-similar set.
\end{defn}

It is a standard fact from Fractal Theory that there exists an atractor for any IFS on a complete metric space.

IFS-attractors can be always equipped with a natural fractal structure, which was first sketched in \cite{BA92}, and formally defined later in \cite[Definition 4.4]{MSG12}. 
%Thus, while the latter provides the definition of a fractal structure as a mathematical concept, the former ``talks'' about the (natural) fractal structure of a self-similar set. 
Next, we recall the description of such a fractal structure, which becomes essential onwards.

\begin{defn}\label{def:fsifs}
Let $\F$ be an IFS, whose associated IFS-attractor is $K$. The natural fractal structure on $K$ is defined as the countable family of coverings $\ef=\{\Gamma_n\}_{n\in \mathbb{N}}$, where $\Gamma_n=\{f_{\omega}(K):\omega\in I^{n}\}$, for each $n\in \mathbb{N}$. It is worth mentioning that, for a given natural number $n$, and each word $\omega=\omega_1 \ \omega_2\ \ldots \ \omega_n\in I^{n}$, we denote $f_{\omega}=f_{\omega_1}\circ\ \ldots\ \circ f_{\omega_n}$.
\end{defn}

\begin{obs}
Another appropriate description for the levels in such a natural fractal structure can be carried out as follows: $\Gamma_1=\{f_i(K):i\in I\}$, and $\Gamma_{n+1}=\{f_i(A):A\in \Gamma_n,i\in I\}$, for all $n\in \mathbb{N}$.
\end{obs}

%\subsection{About the open set condition}\label{sub:osc}

On the other hand, the \emph{open set condition} (OSC in short) becomes a relevant hypothesis required to the similarities $f_i$ of an Euclidean IFS $\F$, in order to guarantee that the pieces $f_i(K)$ of the corresponding IFS-attractor $K$ do not overlap \emph{too much}. Technically, such a condition is satisfied if and only if there exists a non-empty bounded open subset $V\subset \mathbb{R}^d$, such that $\bigcup_{i\in I}f_i(V)\subset V$, where that union remains disjoint (see, e.g., \cite[Section 9.2]{FAL90}).

\subsection{The classical Moran's Theorem}\label{sub:moran}
In \cite[Theorem III]{MOR46} (or see \cite[Theorem 9.3]{FAL90}), it was provided a quite interesting result which allows the calculation of the box dimension for a certain class of Euclidean self-similar sets through the solution of an easy equation only involving a finite number of quantities, namely, the similarity factors that give rise to its corresponding IFS-attractor. Such a classical result is described next.

%\begin{teo}\label{teo:fal}
%(Moran's Theorem)
%Let $I=\{1,\ldots,k\}$ be a finite index set, and let $(\mathbb{R}^d,\{f_i:i\in I\})$ be an IFS, whose associated attractor is $K$. Let us suppose that $c_i$ is the similarity factor associated with each similarity $f_i$, and let us suppose that all of them satisfy the OSC.
%Then $\h(K)=\bc(K)=s$, where $s$ is given by
%\begin{equation}\label{eq:fal}
%\sum_{i=1}^{k}c_{i}^{s}=1.
%\end{equation}
%Further, for this value of $s$, it is verified that $\mathcal{H}_{H}^{s}(K)\in (0,\infty)$.
%\end{teo}

\begin{namedtheorem}[Moran](1946)
Let $\F$ be an Euclidean IFS satisfying the OSC, whose associated IFS-attractor is $K$. Let us assume that $c_i$ is the similarity factor associated with each similarity $f_i$.
Then $\h(K)=\bc(K)=s$, where $s$ is given by the next expression:
\begin{equation}\label{eq:fal}
\sum_{i=1}^{k}c_{i}^{s}=1.
\end{equation}
Further, for this value of $s$, it is satisfied that $\mathcal{H}_{H}^{s}(K)\in (0,\infty)$.
\end{namedtheorem}

It is worth mentioning that Moran also provided in \cite[Theorem II]{MOR46} a weaker version for the result above, under the assumption that all the similarities are equal.

\section{Counterexamples regarding theoretical properties for a fractal dimension definition}\label{sec:def}

Next theorem, which can be found along \cite[Section 2.2]{FAL90}, contains some analytical properties that are satisfied by our key reference for a fractal dimension definition, namely, the Hausdorff dimension. These properties will be used along this section for comparative purposes regarding our models of fractal dimension for a fractal structure as well as the classical ones, namely, both the Hausdorff and the box dimensions.
%this result hereafter in order to compare these properties for the different models of fractal dimension for fractal structures we will introduce in the following chapters.

\begin{teo}\label{teo:fal}
%(see \cite[Section 2.2]{FAL90A})\label{prop:prhaus}
\begin{enumerate}[(1)]
\item \emph{Monotonicity}: if $E\subseteq F$, then $\dih(E)\leq \dih(F)$.
\item \emph{Finite stability}: $\dih(E\cup F)=\max\{\dih(E), \dim(F)\}$.
\item \emph{Countable stability}: if $\{F_i\}_{i\in I}$ is a countable collection of sets, then
\begin{equation*}\label{eq:cstability}
\dih\Bigg(\bigcup_{i\in I}F_i\Bigg)=\sup\{\dih(F_i):i\in I\}.
\end{equation*}
\item \emph{Countable sets}: if $F$ is a countable set, then $\dih(F)=0$.\label{prop:countable}
\item In general, it is not satisfied that $\dih(F)=\dih(\overline{F})$.\label{prop:countable}
\end{enumerate}
\end{teo}
It is worth mentioning that the latter property, which we will refer to as $\cldim$, herein, would not be desired (at least at a first glance) to be satisfied by any fractal dimension definition. The key reason was given in \cite[Subsection 3.2]{FAL90}. Indeed, if $\dim(F)=\dim(\overline{F})$, for any subset $F$ of $X$, then it turns out that a ``small'' (countable) set of points can \emph{wreak havoc} with the dimension, since it may be non-zero. This constitutes a technical reason that makes the box dimension be seriously limited from a theoretical point of view. A proof for the properties contained in Theorem \ref{teo:fal} regarding the box dimension can be found in \cite[Subsection 3.2]{FAL90}.
Interestingly, as stated in \cite[Chapter 3]{FAL90}, it turns out that all fractal dimension definitions are monotonic, and most of them are finitely stable. However, some common definitions do not satisfy the countable stability property, and even they may throw positive values for the dimensions of certain countable sets. In fact, this is the case of box dimension.

Recall that upcoming Table \ref{table:1} provides a schematic summary regarding the properties from Theorem \ref{teo:fal} that are satisfied by each fractal dimension definition. We would like to point out that the absence of a \OK\ sign in that table has been properly justified by an appropriate counterexample along the present Section \ref{sec:def}.
\begin{table} 
\centering
\begin{center}
\begin{tabular}{@{} cl*{5}c @{}}
        & & \multicolumn{5}{c}{Theoretical properties} \\[2ex]
        & & \rot{Monotonicity} & \rot{F-stability} & \rot{C-stability} & \rot{$0$-countably} 
        & \rot{$\cl$-$\, \dim$} \\
        \cmidrule{2-7}
        & $\bc$             & \OK & \OK  &  &   &  \OK     \\
        & $\uno$             & \OK & \OK  &   &   &       \\
        & $\dos$             & \OK &   &   &   &       \\
        & $\tres$             & \OK & \OK  &   &   &       \\
        & $\cuatro$ & \OK & \OK  &   &   & \OK      \\
% \rot{\rlap{~Fractal dimensions}}
        & $\cinco$           & \OK & \OK  & \OK  & \OK  &       \\
%\rot{\rlap{~Fractal dimension models}}
        & $\seis$             & \OK & \OK  & \OK  &  \OK &       \\
\rot{\rlap{~Fractal dimensions}}
        & $\dih$             & \OK & \OK  & \OK  &  \OK &       \\
%\rot{\rlap{~Fractal dimension models}}
        \cmidrule[1pt]{2-7}
\end{tabular}
\end{center}
\caption{The table above contains those analytical properties that are satisfied by all the fractal dimension models considered along this paper. Note that F-stability refers to the finite stability property for a fractal dimension $\dim$, C-stability refers to the countable stability, $0$-countably means that a fractal dimension $\dim$ is zero for countable subsets, and finally, $\cl$-$\dim$ refers to the following property: $\dim(F)=\dim(\overline{F})$.}\label{table:1}
\end{table}

%\subsection{Regarding the countable stability}\label{sub:11}
%The following counterexample 
Firstly, we would like to highlight that fractal dimensions I, II, III \& IV are not countably stable. In fact, as we will show next, there exist countable Euclidean subsets (equipped with an induced natural fractal structure), whose fractal dimensions I, II, III \& IV are non-zero. Recall that fractal dimensions I and III are someway expected to not satisfy such a property, since their description is similar to the box dimension (in the case of fractal dimension I), or at least, they generalize it in the context of Euclidean spaces equipped with their natural fractal structures (as it happens with fractal dimension III, see \cite[Theorem 4.15]{DIM3}, but also with fractal dimension I, see \cite[Theorem 3.5]{DIM1}). Surprisingly, fractal dimension IV, which matches the Hausdorff dimension for compact Euclidean subsets from a discrete point of view (recall \cite[Theorem 3.13]{DIM4}), does not satisfy such a property, too. 
%Interestingly, this model for a fractal dimension allowed the authors to contribute the first-known overall algorithm for Hausdorff dimension calculation purposes (see \cite[Algorithm 3.1]{HAU15}). 

\begin{ct}\label{ct:0}
%There exists a countable subset $F$ of $X$, such that $\uno(F)\neq 0$, for a certain fractal structure $\ef$. 
There exists a countable subset $F$ of $X$, such that $\dim_{\ef}^{k}(F)\neq 0$, for a certain fractal structure $\ef$, where $k=1,2,3,4$. 
\end{ct}

\begin{proof}
Let us consider $X=[0,1]$, $F=\mathbb{Q}\cap X$, and $\ef$ be the natural fractal structure on $X$, whose levels are defined as in Eq. (\ref{eq:1}). 
%\begin{equation}\label{eq:1}
%\Gamma_{n}=\Big\{\Big[\frac{k}{2^{n}},\frac{k+1}{2^{n}}\Big]: k\in \{0,1,\ldots,2^{n}-1\}\Big\}.
%\end{equation}
\begin{enumerate}[(i)]
\item (See \cite[Proposition 3.6 (4)]{DIM1}). Hence, $N_{n}(F)=2^{n}$, so $\uno(F)=\dos(F)=1$.\label{ct:01}
\item (See \cite[Proposition 4.16 (2)]{DIM3}). Since fractal dimension III generalizes fractal dimension I in the context of Euclidean spaces equipped with their natural fractal structures (see \cite[Theorem 4.15]{DIM3}), then
$\dim_{\ef}^{3}(F)=\dim_{\ef}^{1}(F)=1$, just applying the arguments stated previously for Counterexample \ref{ct:0} (\ref{ct:01}).\label{ct:12}
\item (See \cite[Proposition 3.4 (1)]{DIM4}). $F$ is a countable subset such that $\overline{F}=[0,1]$. Thus, \cite[Theorem 3.13]{DIM4} leads to $\cuatro(F)=\dih(\overline{F})=1$.
\end{enumerate}
\end{proof}

On the other hand, it is worth mentioning that fractal dimension II, which generalizes fractal dimension I, in the sense that it allows additionally that different diameter sets could appear in each level of the fractal structure, not even satisfies the finite stability property (and hence, this cannot be countably stable, too), as the following counterexample states. Thus, it turns out that all the fractal dimension definitions considered along this paper are finitely stable with the exception of fractal dimension II (see Table \ref{table:1}). 
%Next, we provide an appropriate counterexample that allows to show that fractal dimension II is not finitely stable.

\begin{ct}(\cite[Example 4]{DIM1})\label{ct:6}
Neither the lower fractal dimension II nor the upper fractal dimension II are finitely stable.
\end{ct}

\begin{proof}
In fact, let $\ef_{1}$ be the natural fractal structure on $C_{1}$ as an IFS-attractor, where $C_{1}$ is the middle third Cantor set on $[0,1]$. Additionally, let also $\ef_{2}$ be a fractal structure on $C_{2}=[2,3]$, defined as $\ef_{2}=\{\Gamma_{2,n}\}_{n\in \mathbb{N}}$, where 
$$\Gamma_{2,n}=\Big\{\Big[\frac{k}{2^{2n}},\frac{k+1}{2^{2n}}\Big]:k\in \{2^{2n+1},2^{2n+1}+1, \,\ldots\,, 3\, 2^{2n}-1\}\Big\}.$$
Thus, let $\ef=\{\Gamma_{n}\}_{n\in \mathbb{N}}$ be a fractal structure on $C=C_{1}\cup C_{2}$, where its levels are given by $\Gamma_{n}=\Gamma_{1,n}\cup \Gamma_{2,n}$. Finally, simple calculations lead to $\dos(C_{1})=\log 2/ \log 3$, as well as to $\dos(C_{2})=1$, whereas $\dos(C)=\log 4/ \log 3>1$.
\end{proof}

As a consequence of both Counterexamples \ref{ct:0} and \ref{ct:6}, it holds that neither of those fractal dimension models involved therein, namely, fractal dimensions I, II, III, and IV, is countably stable. To deal with, recall that unlike fractal dimensions I and II, it holds that fractal dimensions III and IV do always exist. 
%According to that, the following counterexample establishes that neither of them is countably stable.

\begin{cor}
Neither the (lower/upper) fractal dimensions I, II, nor the fractal dimensions III, IV are countably stable.
\end{cor}

%\begin{proof}
%%Let us consider $X=[0,1]$, $F=\mathbb{Q}\cap X$, and $\ef$ be the natural fractal structure induced on $X$, whose levels are given by 
%%$$\Gamma_{n}=\Big\{\Big[\frac{k}{2^{n}},\frac{k+1}{2^{n}}\Big]: k\in \{0,1,\ldots,2^{n}-1\}\Big\}.$$
%%Thus,
%\begin{enumerate}[(i)]
%\item For (lower/upper) fractal dimension I, the result follows immediately from Counterexample \ref{ct:0} (\ref{ct:01}).\label{ct:21}
%\item For (lower/upper) fractal dimension II, the result yields from Counterexample \ref{ct:6}.\label{ct:21}
%\item (See \cite[Proposition 4.16 (3)]{DIM3}). Let $X=[0,1], F=\mathbb{Q}\cap X$, and $\ef$ be the natural fractal structure on $[0,1]$, whose levels are given as in Eq. (\ref{eq:1}).
%Then $\dim_{\ef}^{3}(F)=\dim_{\ef}^{1}(F)=1>0$, since the fractal dimension III generalizes the fractal dimension I in the context of Euclidean spaces equipped with their natural fractal structures (just apply \cite[Theorem 4.15]{DIM3}). In fact, recall that $N_n(F)=2^n$, for all $n\in \mathbb{N}$.
%Moreover, for all rational number $q_i\in F=\mathbb{Q}\cap X$, let us denote $F_i=\{q_i\}$. Thus, it becomes clear that $\tres(F_i)=0$, for all $i\in \mathbb{N}$, so $\sup\{\tres(F_i):i\in \mathbb{N}\}=0$. On the other hand, 
%$$\tres(F)=\tres\Bigg(\bigcup_{i\in \mathbb{N}}F_i\Bigg)=\tres\Bigg(\bigcup_{q_i\in F}\{q_i\}\Bigg)=1,$$
%which concludes the proof. 
%\end{enumerate}
%\end{proof}

The following counterexample points out that fractal dimensions I, II, III, V, and VI, do not satisfy the $\cldim$ property.

\begin{ct}\label{ct:1}
%There exists a countable subset $F$ of $X$, such that $\uno(F)\neq 0$, for a certain fractal structure $\ef$. 
There exist a subset $F$ of a certain space $X$, as well as a locally finite starbase fractal structure $\ef$, such that $\dim_{\ef}^{k}(F)\neq \dim_{\ef}^{k}(\overline{F})$, for $k=1,2,3,5,6$.
\end{ct}

\begin{proof}

\begin{enumerate}[(i)]
\item (See \cite[Proposition 3.6 (5)]{DIM1}).\label{ct:011}
Let $\ef=\{\Gamma_{n}\}_{n\in \mathbb{N}}$ be a fractal structure defined on $X=([0,1]\times \{0\})\, \bigcup\, \{\{\frac{1}{2^{n}}\}\times [0,1]:n\in \mathbb{N}\}$, whose levels are given by
\begin{equation*}
\begin{split}
\Gamma_{n} & =\Big\{\Big[\frac{k}{2^{n}},\frac{k+1}{2^{n}}\Big]\times \{0\}:k\in \{0,1,\ldots,2^{n}-1\}\Big\}\\
&\bigcup\, \Big\{\Big\{\frac{1}{2^{m}}\Big\}\times \Big[\frac{k}{2^{n}},\frac{k+1}{2^{n}}\Big]:k\in \{0,1,\ldots,2^{n}-1\}, m\in \mathbb{N}\Big\}.
\end{split}
\end{equation*}
(see Fig. (\ref{fig:nowaks}) for a sketch regarding the first level of that fractal structure).
\begin{center}
\begin{figure}
\centering
\includegraphics[width=80mm, height=60mm]{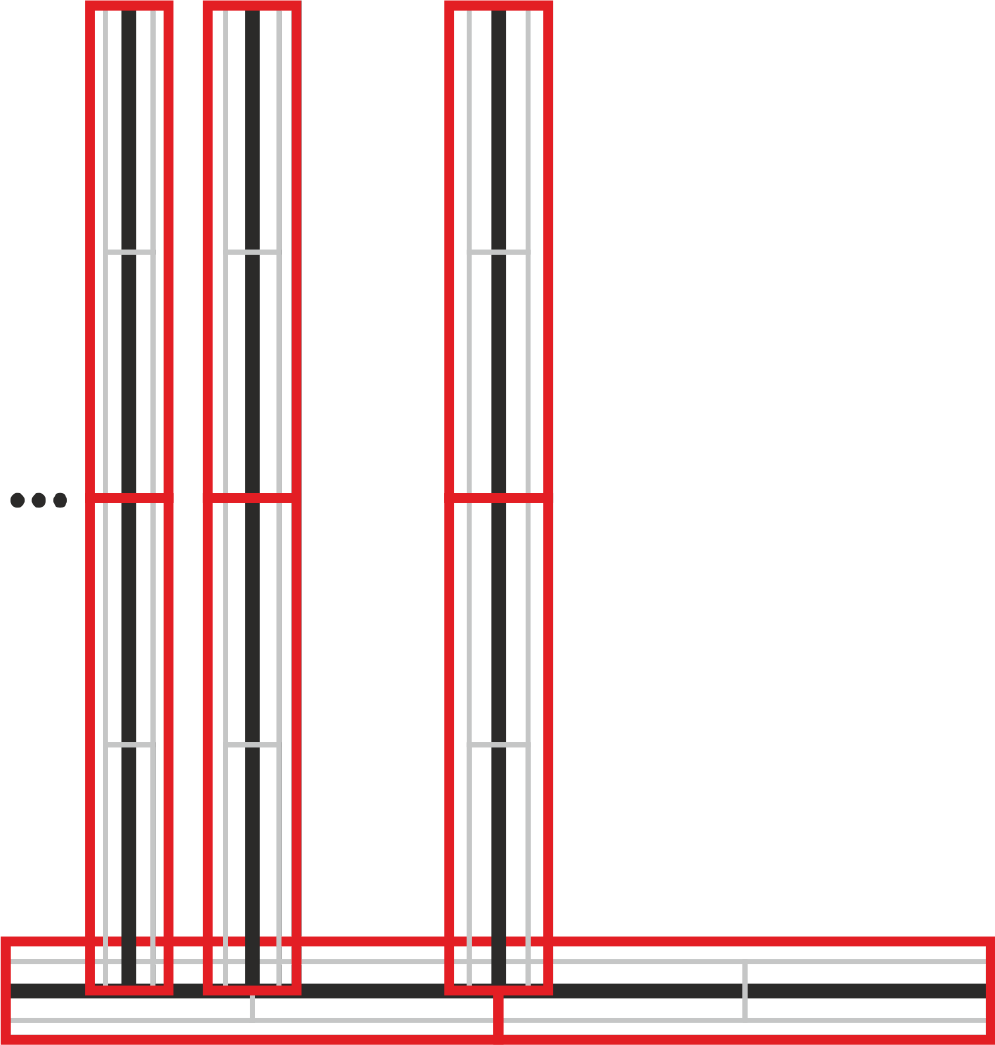} 
\caption{Sketch regarding the first level in the construction of the fractal structure appeared in Counterexample \ref{ct:1} (\ref{ct:011}).}\label{fig:nowaks}
\end{figure}
\end{center}
Moreover, let $F=\bigcup_{k\in \mathbb{N}}(\frac{1}{2^{k+1}},\frac{1}{2^{k}})\times \{0\}$, be a subset of $X$. Thus, $\overline{F}=[0,1]\times \{0\}$, which implies that $N_{n}(F)=2^{n}$, and $N_{n}(\overline{F})=\infty$. Hence, $\uno(F)=1$, and $\uno(\overline{F})=\infty$, respectively. These arguments also become valid in order to justify that fractal dimension II does not satisfy the $\cldim$ property, too. In fact, recall that fractal dimension II generalizes fractal dimension I (in the sense of \cite[Theorems 4.6 \& 4.7]{DIM1}), so any counterexample valid for fractal dimension I, still remains valid for fractal dimension II.
\item (See \cite[Proposition 4.16 (4)]{DIM3}).
Let us consider the same list $(F,X,\ef)$, as provided in Counterexample \ref{ct:1} (\ref{ct:011}).
Thus, $\dim_{\ef}^{3}(F)=1$, whereas
$\mathcal{H}_{n}^{s}(\overline{F})=\sum_{i=0}^{2^n-1}\frac{1}{2^{ns}}+\sum_{i=1}^{\infty}\frac{1}{2^{ns}}=\infty$, which implies that $\dim_{\ef}^{3}(\overline{F})=\infty$.
\item (See \cite[Proposition 3.4 (4)]{DIM4}).
Let $\ef$ be the natural fractal structure on $X=[0,1]$, namely, $\ef=\{\Gamma_n\}_{n\in \mathbb{N}}$, whose levels are defined as 
%$\Gamma_n=\{[\frac{k}{2^n},\frac{k+1}{2^n}]:k\in \{0,1,\ldots,2^n-1\}\}$ 
\begin{equation}\label{eq:1}
\Gamma_{n}=\Big\{\Big[\frac{k}{2^{n}},\frac{k+1}{2^{n}}\Big]: k\in \{0,1,\ldots,2^{n}-1\}\Big\},
\end{equation}
%for all natural number $n$, 
and let $F=\mathbb{Q}\cap X$. Thus, $F$ is a countable subset such that $\overline{F}=[0,1]$. Moreover, since fractal dimensions V, VI are countably stable, due to \cite[Proposition 3.4 (3)]{DIM4}), then $\cinco(F)=\seis(F)=0$. However, $\cinco(\overline{F})=\dih(\overline{F})=\seis(\overline{F})=1$, since both fractal dimensions V, VI do generalize the Hausdorff dimension on a Euclidean subspace with the natural fractal structure (see \cite[Corollary 3.11]{DIM4}).
\end{enumerate}
\end{proof}

\section{Some differences between fractal dimensions I \& II}\label{sec:dim1}

Recall that fractal dimension I model, which was introduced previously in Definition \ref{def:typebox} (\ref{def:dim1}), actually considers all the elements in each level of a given fractal structure as having the same ``size", equal to $1/2^{n}$. According to that, the natural fractal structure which any Euclidean space can be equipped with, allows to calculate the box dimension through another kind of tilings, such as triangulations on the plane, for instance. In particular, this fact becomes quite interesting since any compact surface always owns a triangulation.

On the other hand, if in addition to the fractal structure, a distance function is available in the space, then this could be applied to ``measure'' the size of the elements in each level of the fractal structure at the same time we use the fractal structure. This is the case, for instance, of any Euclidean space, where we can always consider both the natural fractal structure as well as the Euclidean metric. 
%In general, the new definition of fractal dimension that we provide next is formulated in terms of a generic distance function.

The next counterexample points out that fractal dimension I depends on a fractal structure, whereas the upcoming remark establishes that fractal dimension II also depends on a metric. 

\begin{ct}(\cite[Remark 3.11]{DIM1})\label{ct:4}
Fractal dimension I depends on the fractal structure we choose in order to calculate it. Mathematically, let $X$ be a subspace of an Euclidean space.
Then there exist a subset $F$ of $X$, as well as two different fractal structures, $\Gamma_1$ and $\Gamma_2$, such that $\dim_{\ef_1}^1(F)\neq \dim_{\ef_2}^1(F)$. 
%Let $X$ be a subspace of an Euclidean space. Then it is possible to obtain different values for its fractal dimension I depending on the fractal structure we select in order to calculate it.
%There exist a subspace $F$ of an Euclidean space and two (different) fractal structures $\Gamma_1$ and $\Gamma_2$ such that $\dim_{\Gamma_1}^1(F)\neq \dim_{\Gamma_2}^1(F)$. 
\end{ct}

\begin{proof}
Firstly, let $\ef_{1}$ be the natural fractal structure on the middle third Cantor set $C$. Thus, by \cite[Theorem 3.5]{DIM1}, $\dim_{\ef_{1}}^{1}(C)=\dim_{B}(C)=\log 2/ \log 3$, where the second equality holds by \cite[Example 3.3]{FAL90}. On the other hand, let $\ef_{2}$ be the natural fractal structure on $C$ as a self-similar set (recall Definition \ref{def:fsifs}). Then easy calculations lead to $\dim_{\ef_{2}}^{1}(C)=1$, since in each level $\Gamma_{2,n}$ of the fractal structure $\ef_{2}$, there are $2^{n}$ ``subintervals" whose lengths are equal to $1/3^n$.
\end{proof}

\begin{obs}(\cite[Remark 4.9]{DIM1})
In Counterexample \ref{ct:4}, it was shown that the fractal dimension I and the box dimension for the middle third Cantor set $C$ are not equal. In fact, recall that $\bc(C)=\log 2/ \log 3$, whereas $\uno(C)=1$. Notice that such fractal dimensions have been calculated with respect to different fractal structures. In fact, while the natural fractal structure on $C$ as an Euclidean subset has been applied to calculate the box dimension, the natural fractal structure on $C$ as a self-similar set has been chosen for fractal dimension I calculation purposes.
However, if the natural fractal structure on $C$ as a self-similar set is chosen again, then simple calculations allow to prove both the fractal dimension II and the box dimension for $C$ are equal:
$$\dos(C)=\lim_{n\to \infty}\frac{\log 2^{n}}{-\log 3^{-n}}=\frac{\log 2}{\log 3}=\bc(C),$$
since in level $n$ of that fractal structure, there are $2^{n}$ ``subintervals" whose diameters are equal to $1/3^{n}$.
\end{obs}
According to that, the fractal dimension II value for the middle third Cantor set (equipped with its natural fractal structure as an Euclidean subpace) agrees with the box dimension for such a self-similar set.
Even more, though the value obtained for the fractal dimension I of $C$ may seem counterintuitive at a first glance, it still becomes possible to justify it through its fractal dimension II. Once again, the key reason lies in the advanced fact that fractal dimension I only depends on the chosen fractal structure. This is emphasized in the following remark. 

\begin{obs}(\cite[Remark 4.10]{DIM1})\label{obs:40}
Fractal dimension I only depends on a fractal structure, whereas fractal dimension II also depends on the (``maximum'') diameter of the elements in each level of that fractal structure. To show this underlying difference, we construct a family of spaces such that from the point of view of fractal structures are the same.
\end{obs}

\begin{proof}
To deal with, let us consider slight modifications regarding the standard middle third Cantor set $C$, which we will refer to as $C_{i}$. Further, let us assume that their associated similarity factors are $c_{i}\in [\frac{1}{3},\frac{1}{2})$, for each of the two similarities that give rise to $C_{i}$. Hence, $\delta(C_{i},\Gamma_{n})=c_{i}^{n}$, for all $n\in \N$. In addition to that, let us consider the natural fractal structure on each space $C_{i}$ as a self-similar set, denoted by $\ef_{i}:i=1,2$. Then easy calculations lead to (or apply \cite[Theorem 4.19]{DIM1}, as well):
$$\bc(C_{i})=\dos(C_{i})=\frac{\log 2}{-\log c_{i}}\longrightarrow 1=\uno(C),$$
wherever $c_{i}\to 1/2$.
\end{proof}

\section{Justifying some requirements to several fractal dimension definitions regarding fractal structures}\label{sec:defs}
In this section, we provide some counterexamples in order to justify some properties, than being quite natural to be assumed regarding the levels of a fractal structure, have to be required to guarantee that fractal dimensions III, IV and V, do behave like the Hausdorff dimension (see Fig. (\ref{fig:1}) at the end of this section).
%, that is the key theoretical model in which these definitions of fractal dimension for a fractal structure have been inspired by. 
This has been carried out along forthcoming Subsection \ref{sub:def2}. Previously, in Subsection \ref{sub:def1}, we explain why the expression $\delta(F,\Gamma_n)$ should be used for fractal dimension II calculation purposes, instead of $\delta(\Gamma_n)$, which would be, at least a first glance, another valid option.

\subsection{Why $\delta(F,\Gamma_n)$ is used for fractal dimension II instead of $\delta(\Gamma_n)$.}\label{sub:def1}

As it was stated previously, the fractal dimension II for any subset $F$ of $X$, ``measures'' the size of the elements in any level of a fractal structure through a distance function. 
%in the quantity $\delta(F,\Gamma_{n})$ to ``measure" the size of the elements of each level of the fractal structure. 
This is done by means of $\delta(F,\Gamma_n)=\sup\{\dm(A):A\in \anf\}$, where $\anf=\{A\in \Gamma_n:A\cap F\neq \mpty\}$. 
%Observe that this consider the size of the elements in each level $n$ through the ``maximum'' of the diameters of all the elements in $\anf$.
On the other hand, it seems that $\delta(\Gamma_{n})=\sup\{\dm(A):A\in \Gamma_n\}$ could be considered, at least at a first glance, in Definition \ref{def:typebox} (\ref{def:dim2}), instead of $\delta(F,\Gamma_{n})$. However, $\delta(F,\Gamma_{n})$ yields a wider applied dimension than $\delta(\Gamma_{n})$, as it is shown next.

%To show that, let $\ef$ be a finite fractal structure defined on a distance space $(X,\rho)$, where $\rho$ is a not-bounded distance function. Hence, $\delta(\Gamma_{n})=\infty$. Consequently, that model would not allow to properly calculate the fractal dimension II for $X$. Additional details regarding this fact are provided along the following counterexample.

\begin{ct}
There exist a bounded Euclidean subset $F$, a fractal structure $\ef$ on $\R^d$, and a natural number $n_0$, such that $\delta(\Gamma_n)=\infty$, for all $n\in \N$, whereas $\delta(F,\Gamma_n)<\infty$, for all $n\geq n_0$.
\end{ct}

\begin{proof}
In fact, let $\ef=\{\Gamma_n\}_{n\in \N}$, be a finite fractal structure defined on the Euclidean space $\R^{d}$, whose levels are defined by
\begin{equation*}
\begin{split}
%\Gamma_{n}&=\Bigg\{\Bigg[\frac{k_{1}}{2^{n}},\frac{k_{1}+1}{2^{n}}\Bigg]\times \ldots \times \Bigg[\frac{k_{d}}{2^{n}},\frac{k_{d}+1}{2^{n}}\Bigg]:k_{i}\in \{-n2^{n}, \ldots, n2^{n}-1\},i\in \{1,\ldots,d\}\Bigg\}\\
%& \cup \{\R^{d}\setminus (-n,n)^{d}\},
\Gamma_{n}&=\Bigg\{\Bigg[\frac{k_{1}}{2^{n}},\frac{k_{1}+1}{2^{n}}\Bigg]\times\ \ldots\ \times \Bigg[\frac{k_{d}}{2^{n}},\frac{k_{d}+1}{2^{n}}\Bigg]:k_1,\,\ldots\,,k_d\in \{-n2^{n}, \ldots, n2^{n}-1\}\Bigg\}\\
& \cup \{\R^{d}\setminus (-n,n)^{d}\}.
\end{split}
\end{equation*}
Hence, for any bounded subset $F$ of $\R^d$, it holds that $\delta(\Gamma_{n})=\infty$, for all $n\in \N$, where such a ``diameter'' has been calculated with respect to the Euclidean distance. However, there exists a natural number $n_{0}$, such that $\delta(F,\Gamma_{n})<\infty$, for all $n\geq n_{0}$.
\end{proof}

\subsection{Why\, $\delta(F,\Gamma_n)$ must be a $0$-convergent decreasing sequence in several Hausdorff type models of fractal dimension for a fractal structure}\label{sub:def2}

Next, we point out that the ``natural'' assumption consisting of $\delta(F,\Gamma_n)\to 0$ becomes necessary to guarantee that some Hausdorff type fractal dimensions for a fractal structure, do indeed behave like that classical model.

\begin{ct}\label{obs:4}
If $\delta(F,\Gamma_{n})\to 0$, then $\inf\{s:\mathcal{H}_k^{s}(F)=0\} =\sup\{s:\mathcal{H}_k^{s}(F)=\infty\}$, for $k=3,4,5$. 
%The converse is not true.
Otherwise, that equality cannot be guaranteed, in general.
\end{ct}

\begin{proof}
In fact, let $\ef=\{\Gamma_n\}_{n \in \N}$, be a fractal structure defined on $X=[0,1]$, whose levels are given as 
%$\ef=\{\Gamma_n:n \in \N\}$, where 
$$\Gamma_n=\Bigg\{\Bigg[0,\frac{1}{2}\Bigg],\Bigg[\frac{1}{2},1\Bigg]\Bigg\} \bigcup \Bigg\{\Bigg[\frac{k}{2^n},\frac{k+1}{2^n}\Bigg] :k\in \{1,\ldots, 2^n-1\}\Bigg\}.$$
Then the four following hold for all $n\in\N$:
\begin{itemize}
\item $\delta(X,\Gamma_n)=1/2$. 
\item Moreover, 
$$\mathcal{H}_{n,3}^s(X)=\inf \Bigg\{\frac{2}{2^s}+(2^{m}-1)\,\frac{1}{2^{ms}}:m \geq n\Bigg\}.$$ 
Hence, 
%$\mathcal{H}_{3}^s(X) = 2^{1-s}$, for $s>1$, whereas $\mathcal{H}_{3}^s(X) = \infty$, for $s<1$.
\begin{equation*}
\mathcal{H}_{3}^s(X)=
\begin{cases} 
\infty & \mbox{if } s<1 \\ 
2^{1-s} & \mbox{if } s>1. 
\end{cases} 
\end{equation*}
\item On the other hand, observe that any cover of $X$ through elements of $\Gamma_n$, for some $n \in \mathbb{N}$, must contain the set $[0,\frac{1}{2}]$. Accordingly, $\mathcal{H}_{n,4}^s(X) \geq \frac{1}{2^s}$, for each $n \in \N$, so $\mathcal{H}_{4}^s(X) \geq \frac{1}{2^s}$. Thus, for $s>1$, it holds that $\mathcal{H}_{4}^s(X) = \frac{1}{2^s}$ (since for the natural fractal structure, $\mathcal{H}_{4}^s(X) = 0$), whereas for $s<1$, $\mathcal{H}_{4}^s(X) = \frac{2}{2^s}$ (since we can use the cover $\{ [0,\frac{1}{2}],[\frac{1}{2},1] \}$).
\item Similar arguments to those applied to deal with the fractal dimension IV case allows us to affirm that
%shows that  that for $s>1$, $\mathcal{H}_{5}^s(X) = \frac{1}{2^s}$, while for $s<1$, $\mathcal{H}_{5}^s(X) = \frac{2}{2^s}$.
\begin{equation*}
\mathcal{H}_{5}^s(X)=
\begin{cases} 
2^{1-s} & \mbox{if } s<1 \\ 
2^{-s} & \mbox{if } s>1. 
\end{cases} 
\end{equation*}
\end{itemize}
\end{proof}

However, for $k=6$, it holds that the condition $\delta(F,\Gamma_{n})\to 0$ is not required to reach the identity $\inf\{s:\mathcal{H}_k^{s}(F)=0\} =\sup\{s:\mathcal{H}_k^{s}(F)=\infty\}$, though it may be desirable.
In fact, for instance, in Counterexample \ref{obs:4}, it holds that any cover of $[0,1]$ by elements of level $n$, for some $n \in \mathbb{N}$, must contain the subinterval $[0,\frac{1}{2}]$. Thus, $\mathcal{A}_{\delta,6}(X)=\emptyset$, for $\delta<\frac{1}{2}$, so $\mathcal{H}_{6}^s(X)=\infty$, for each $s \geq 0$. Therefore, $\seis(X)=\infty$. 

\begin{lema}\label{lema:6}
$\inf\{s:\mathcal{H}_6^{s}(F)=0\} =\sup\{s:\mathcal{H}_6^{s}(F)=\infty\}$.
\end{lema}

\begin{proof}
%Let $t$ be a positive real number and let $\mathcal{C} \in \mathcal{A}_{\delta,6}(F)$. 
Fix $\delta>0$. Let $\mathcal{C}\in \mathcal{A}_{\delta,6}(F)$, and $t$, $s$ be positive real numbers. 
Then it holds that
\begin{equation}\label{eq:ineq}
\sum_{A\in\mathcal{C}} \dm(A)^{t}\leq \delta^{t-s}\, \sum_{A\in\mathcal{C}} \dm(A)^{s}.
\end{equation}
%where the sums are considered for all $A\in\mathcal{A}_{n}(F)$, and 
%where $\delta_{n}:=\delta(F,\Gamma_{n})$, for all $n\in \mathbb{N}$. 
Hence,
\begin{equation}\label{eq:ineq1} 
\mathcal{H}_{\delta,6}^{t}(F)\leq \delta^{t-s} \ \mathcal{H}_{\delta,6}^{s}(F).
\end{equation}
Moreover, if we take limits as $\delta \to 0$ in Eq. (\ref{eq:ineq1}), then
%In this way, if we take limits as $n\to \infty$ on the previous expression then,
$$\mathcal{H}_6^{t}(F)\leq \mathcal{H}_6^{s}(F)\, \lim_{\delta \to 0}\delta^{t-s}.$$
Accordingly, if it is assumed that $\mathcal{H}_6^{s}(F)<\infty$, provided that $t>s$, then $\mathcal{H}_6^{t}(F)=0$.
Therefore, $\seis(F)=\inf\{s:\mathcal{H}_6^{s}(F)=0\}=\sup\{s:\mathcal{H}_6^{s}(F)=\infty\}$.
\end{proof}

\begin{center}
\begin{figure*}[here]
\centering
\begin{tabular}{c}
\includegraphics[width=120mm, height=60mm]{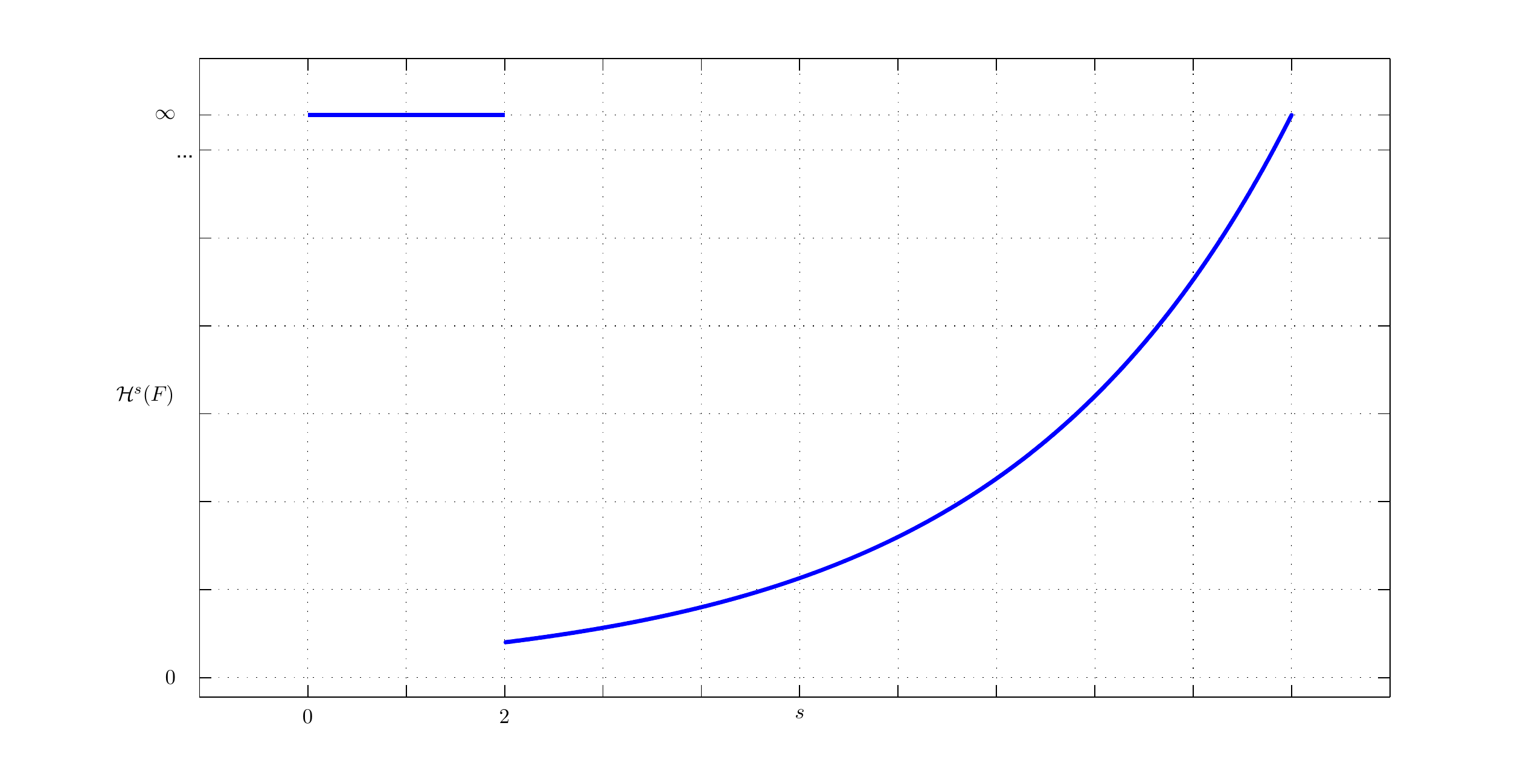} \\ \includegraphics[width=120mm, height=60mm]{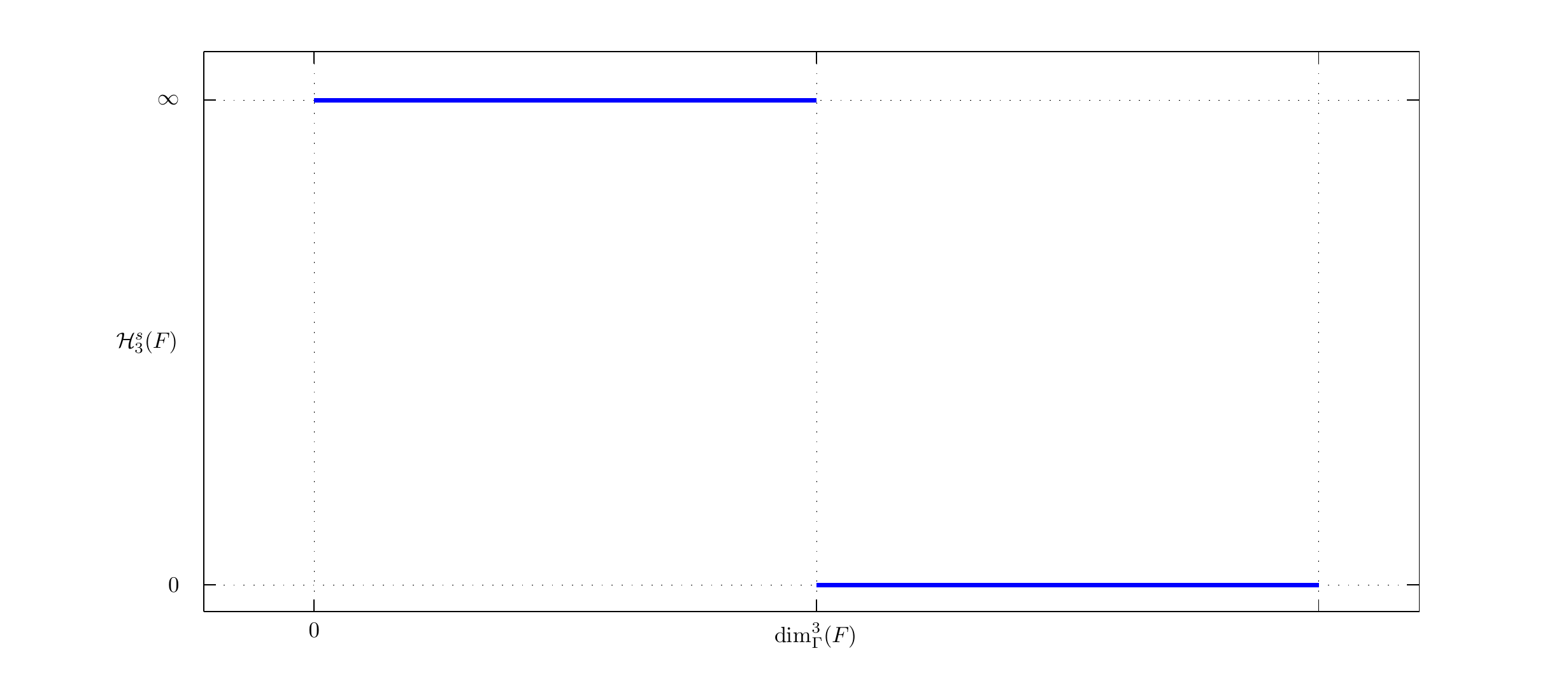}
\end{tabular}
\caption{Graphical representation of $s$ vs. $\mathcal{H}_3^{s}(F)$, where $F=[0,1]\times [0,1]$, and $\ef$ is its natural fractal structure (note that $\delta(F,\Gamma_{n})\nrightarrow 0$) (above), and plot of $s$ vs. $\mathcal{H}_{k}^{s}(F):k=3,4,5$, under the assumption $\delta(F,\Gamma_n)\to 0$, which, however, is not required for fractal dimension VI (see Lemma \ref{lema:6}).}\label{fig:1}
\end{figure*}
\end{center}

\section{Counterexamples involving fractal dimensions for IFS-attractors}\label{sec:ifs}

The main purpose in this section is to emphasize that some hypothesis required in several Moran type theorems are necessary. These theoretical results, that we will recall next, have been proved for fractal dimensions developed in the context of fractal structures (recall Subsection \ref{sub:models}). In this way, it is also worth mentioning that some of these models allow to generalize classical fractal dimensions to fractal structures.

%It is worth mentioning that, in \cite[Theorem III]{MOR46} (or see \cite[Theorem 9.3]{FAL90}), it was provided a quite interesting result which allows the calculation of the box dimension of a certain class of Euclidean self-similar sets through the solution of an easy equation involving only a finite number of quantities, namely, the similarity factors that give rise to its corresponding IFS-attractor. Such a classical result is described next.

To deal with, firstly, recall that Moran's Theorem (see Subsection \ref{sub:moran}) allows to easily calculate both the Hausdorff and the box dimensions for strict self-similar sets coming from Euclidean IFS satisfying the OSC. In fact, such a classical result throws both Hausdorff and box dimensions as the solution of an equation which only involves the similarity factors that give rise to the corresponding IFS-attractor.

Additionally, it is worth mentioning that the result provided below allowed its authors to reach the equality between both box dimension and fractal dimension II for strict self-similar sets whose IFS are under the OSC. Moreover, the calculation of such quantity is immediate from the number of similarities in the IFS and their common similarity factor.

\begin{teo}(\cite[Theorem 4.19]{DIM1})\label{teo:8}
Let $\F=\{f_i:i\in I\}$ be an Euclidean IFS satisfying the OSC, whose associated IFS-attractor is $K$. Moreover, let us assume that all the similarities in $\F$ have a common similarity factor, $c\in (0,1)$.
Thus, if $\ef$ is the natural fractal structure on $K$ as a self-similar set, then
\begin{equation}\label{eq:13}
\bc(K)=\dos(K)=\frac{-\log k}{\log c}.
\end{equation}
\end{teo}

Next, we point out that the hypothesis regarding the equality among the similarity factors in Theorem \ref{teo:8} becomes necessary. 

\begin{ct}\label{ct:10}
There exists an Euclidean IFS $\F$ satisfying the OSC, having different contraction factors, and whose associated IFS-attractor $K$, which is equipped with its natural fractal structure as a self-similar set, satisfies that $\bc(K)<\dos(K)$.
\end{ct}

\begin{proof}
Let $\F$ be an Euclidean IFS, whose similarities $f_{1},f_{2}:\mathbb{R}\lto \mathbb{R}$ are defined by
$$f_{i}(x)=\left\{
\begin{array}{ll}
\hbox{$\frac{x}{2}$} & \hbox{if $i=1$;} \\
\hbox{$\frac{x+3}{4}$} & \hbox{if $i=2$.}
\end{array}
\right.$$
Observe that their associated contraction factors are $c_{1}=1/2$, and $c_{2}=1/4$, respectively. Moreover, it holds that $K$ is a strict self-similar set. Further, we can also justify that the IFS $\F$ satisfies the OSC. In fact, just take $V=(0,1)$, as an appropriate open subset. Hence, the Moran's Theorem allows to affirm that the box dimension of $K$ is the solution of the equation $\frac{1}{2^{s}}+\frac{1}{4^{s}}=1$. Thus, $\bc(K)=\log(\frac{1+\sqrt{5}}{2})/\log 2$. On the other hand, there are $2^{n}$ subintervals of $[0,1]$, in each level $n$ of the fractal structure $\ef$, where the largest of them has a diameter equal to $1/2^n$. Accordingly, $\dos(K)=1>\bc(K)$.  
\end{proof}

In the next theorem, we explored which properties regarding the natural fractal structure on any Euclidean space could allow to generalize \cite[Theorem 4.19]{DIM1}. In this way, it is worth mentioning that, fixed a scale $\delta>0$, then it is satisfied that any Euclidean subset $F$ of $\R^d$, with $\dm(F)\leq \delta$, intersects at most $3^{d}$ $\delta$-cubes. Given this, a similar property in the more general context of fractal structures allowed the authors to find out an additional link between fractal dimension II and box dimension.

\begin{teo}(\cite[Theorem 4.13]{DIM1})\label{teo:7}
Let $F$ be a subset of a metric space $(X,\rho)$, and let $\ef$ be a fractal structure on $X$.
%Let $F$ be a subset of $X$, and $\ef$ be a fractal structure on a metric space $(X,\rho)$. 
Moreover, assume that there exists $k\in \N$, such that, for all $n\in \N$, every subset $A$ of $X$, with $\dm(A)\leq \delta(F,\Gamma_{n})$, intersects at most $k$ elements in level $n$ of\, $\ef$. Additionally, let us suppose also that $\delta(F,\Gamma_{n})\to 0$. Then,
\begin{enumerate}[(1)]
\item $\underline{\dim}_{B}(F)\leq \underline{\dim}_{\ef}^{2}(F)\leq \overline{\dim}_{\ef}^{2}(F)\leq \overline{\dim}_{B}(F).$
Further, if there exists $\bc(F)$, then $\bc(F)=\dos(F)$.
\item If there exists a constant $c\in (0,1)$, such that
$c\, \delta(F,\Gamma_n)\leq \delta(F,\Gamma_{n+1})$,
%$\delta(F,\Gamma_{n+1})\geq c\cdot \delta(F,\Gamma_{n})$,
then $\overline{\dim}_{B}(F)=\overline{\dim}_{\ef}^{2}(F)$, and $\underline{\dim}_{B}(F)=\underline{\dim}_{\ef}^{2}(F)$.
\end{enumerate}
\end{teo}

In this way, the following counterexample establishes the need for the main hypothesis required in Theorem \ref{teo:7} in order to reach the equality between both the box dimension and the fractal dimension II.
Accordingly, while Counterexample \ref{ct:10} allows to guarantee that all the similarity factors must be the same, the counterexample below shows that all the contractions in $\F$ must be similarities, as well.

\begin{ct}\label{ct:9}
There exists an Euclidean IFS $\F$, whose associated IFS-attractor $K$, equipped with its natural fractal structure $\ef$ as a self-similar set, satisfies that $$\bc(K)\neq \dos(K).$$
\end{ct}

\begin{proof}
Let $I=\{1,\ldots,8\}$ be a finite index set, and $(\R^{2},\F)$ be an Euclidean IFS, whose associated IFS-attractor is $K=[0,1]\times [0,1]$. In addition to that, let us define the contractions $f_{i}:\R^{2}\lto \R^{2}$, as follows:
$$f_{i}(x,y)=\left\{
\begin{array}{ll}
\hbox{$(\frac{x}{2},\frac{y}{4})+(0,\frac{i-1}{4})$} & \hbox{if $i=1,2,3,4$;} \\
\hbox{$(\frac{x}{2},\frac{y}{4})+(\frac{1}{2},\frac{i-5}{4})$} & \hbox{if $i=5,6,7,8$.}
\end{array}
\right.$$
Moreover, let $\ef$ be the natural fractal structure on $K$ as a self-similar set. First of all, notice that the self-maps $f_{i}$ are not similarities but affinities, and that all the mappings have the same contraction factor, namely, $c_{i}=1/2$, for all $i\in I$. Further, it also becomes immediate that $\bc(K)=2$. On the other hand, note that there are $8^{n}$ rectangles in level $n$ of $\ef$, whose dimensions are $\frac{1}{2^{n}}\times \frac{1}{2^{2n}}$. Thus, it is satisfied that 
$$\dm(A)=\delta(K,\Gamma_{n})=\sqrt{\frac{1+2^{2n}}{2^{4n}}},$$ 
for all $A\in \Gamma_{n}$. Hence,
$$\dos(K)=\lim_{n\to \infty}\frac{\log N_{n}(K)}{-\log \delta(K,\Gamma_{n})}=
\lim_{n\to \infty}\frac{3n\log 2}{-\frac{1}{2}\log \frac{1+2^{2n}}{2^{4n}}}=\lim_{n\to \infty}\frac{3n\log 2}{n\log 2}=3.$$
In addition to that, let us find out the ratio between $\delta(K,\Gamma_{n})$ and each side of any $\frac{1}{2^{n}}\times \frac{1}{2^{2n}}$-rectangle. In fact, it holds that $\frac{\sqrt{\frac{1+2^{2n}}{2^{4n}}}}{\frac{1}{2^{2n}}}=\sqrt{1+2^{2n}}>2^{n}$, and that $\frac{\sqrt{\frac{1+2^{2n}}{2^{4n}}}}{\frac{1}{2^{n}}}=\sqrt{1+\frac{1}{2^{2n}}}\geq \frac{1}{2^{n}}$, for all $n\in \N$. Therefore, each subset $A$ of $K$, whose diameter is at most equal to $\sqrt{\frac{2^{2n}+1}{2^{4n}}}$, intersects at most $3\, 2^{n+1}$ elements $A\in \Gamma_{n}$. Consequently, given that such a quantity depends on each $n\in \N$, then we conclude that the key hypothesis in Theorem \ref{teo:7} is not satisfied.
\end{proof}

In \cite{DIM3}, the authors contributed an interesting theorem in the line of the Moran's classical one, which allowed them to calculate the so-called fractal dimension III as the solution of an easy equation which only involves the similarity factors associated with each similarity in the corresponding IFS. The main advantage of such a result lies in the fact that the OSC is not required to be satisfied by that IFS, in order to apply that theoretical result. However, if the OSC is also satisfied, then such a fractal dimension equals both box and Hausdorff dimensions. It is worth mentioning that these results make use of the natural fractal structure that each self-similar set can be equipped with (recall Definition \ref{def:fsifs}). Next, we recall that result.

\begin{teo}(\cite[Theorem 4.20]{DIM3})\label{teo:dim3}
Let $\F$ be an IFS on a complete metric space, whose associated IFS-attractor is $K$. In addition to that, let $c_i$ be the similarity factor associated with each similarity $f_i\in \F$, and assume that $\ef$ is the natural fractal structure on $K$ as a self-similar set. Then, it is satisfied that 
$$\sum_{i\in I}c_i^{\tres(K)}=1.$$
Furthermore, it is also satisfied that $\mathcal{H}_3^{\tres(K)}(K)\in (0,\infty)$.
\end{teo}

We would like to point out, through the following counterexample, that the hypothesis regarding the strict self-similarity of the IFS-attractor in Theorem \ref{teo:dim3}, still becomes necessary.

\begin{ct}(\cite[Remark 4.21]{DIM3})\label{ct:12}
There exists an Euclidean IFS $\F$, whose associated (non-strict) IFS-attractor $K$, which is equipped with its natural fractal structure as a self-similar set, satisfies that $s\neq \tres(K)$, where $s$ is the solution of the equation $\sum_{i\in I}c_i^s=1$, and $c_i$ is the contraction factor associated with each contraction $f_i\in \F$.
\end{ct}

\begin{proof}
Let $I=\{1,\ldots,8\}$ be a finite index set, $\F=\{f_i:i\in I\}$, and $(\mathbb{R}^2,\F)$ be an Euclidean IFS, whose associated IFS-attractor is $K=[0,1]\times [0,1]$. Moreover, let $f_i:\mathbb{R}^2\lto \mathbb{R}^2$ be the contractions given by
\begin{equation*}
f_{i}(x,y)=\left\{
\begin{array}{ll}
\hbox{$(\frac{-y}{2},\frac{x}{4})+(\frac{1}{2},\frac{i-1}{4})$} & \hbox{if $i=1,2,3,4$;} \\
\hbox{$(\frac{-y}{2},\frac{x}{4})+(1,\frac{i-5}{4})$} & \hbox{if $i=5,6,7,8$.}
\end{array}
\right.
\end{equation*}
First of all, observe that the self-similar set $K$ is not strict. Further, it holds that the contractions $f_i$ are compositions of affine maps: rotations, dilations (in the plane and with respect to one coordinate), and translations. Moreover, it is also clear that all the contractions $f_i$ have the same contraction factor, equal to $1/2$. Thus, $s=3$ is the solution of $\sum_{i\in I}c_i^s=1$.\newline
On the other hand, we affirm that $\tres(K)=2$. To show that, let us calculate the fractal dimension III of $K$ through \cite[Theorem 4.17]{DIM3}. In fact, let us start by analyzing the even levels in $\ef$. Thus, for all $n\in \mathbb{N}$, it holds that every level $\Gamma_{2n}$ consists of squares having sides equal to $1/8^n$. Moreover, it is also clear that $\dm(A)=\delta(K,\Gamma_{2n})=\sqrt{2}/8^n$, for all $A\in \Gamma_{2n}$. This implies that $\delta(K,\Gamma_{2n})\to 0$. Now, let us check the main hypothesis in \cite[Theorem 4.17]{DIM3}. Thus, let us calculate the maximum number of elements in $\Gamma_{2n}$ which are intersected by a subset $B$, whose diameter is at most equal to $\sqrt{2}/8^n$. In fact, the ratio between the diameter of each square in $\Gamma_{2n}$ and its side is $\sqrt{2}<2$, which allows to state that the number of elements in $\mathcal{A}_{2n}(B)$ is at most $3$ in each direction, for all subset $B$ having $\dm(B)<\delta(K,\Gamma_{2n})$. Accordingly, we can choose $k=9$ as a suitable constant for the even order levels in the fractal structure $\ef$.\newline
Similarly, it can be checked that all the odd order levels $\Gamma_{2n+1}$ consist of rectangles whose dimensions are $\frac{1}{2\cdot 8^n}\times \frac{1}{4\cdot 8^n}$. Further, observe that all the elements in each level $\Gamma_{2n+1}$ have the same diameter, equal to $\frac{\sqrt{5}}{4\cdot 8^n}$. Hence, the sequence of diameters $\delta(K,\Gamma_{2n+1})$ also goes to $0$. To check that, the following ratios between the diameter and the sides of each rectangle yields: $\frac{\frac{1}{4}\cdot \frac{\sqrt{5}}{8^n}}{\frac{1}{2}\cdot \frac{1}{8^n}}=\frac{\sqrt{5}}{2}<2$, and $\frac{\frac{1}{4}\cdot \frac{\sqrt{5}}{8^n}}{\frac{1}{4}\cdot \frac{1}{8^n}}=\sqrt{5}<3$, respectively. Therefore, each subset $A$ whose diameter is at most equal to $\delta(K,\Gamma_{2n+1})$, has to intersect, at most, to $k=12$ elements in each odd level $\Gamma_{2n+1}$.\newline
Consequently, the main hypothesis in \cite[Theorem 4.17]{DIM3} is satisfied. 
%(in fact, just take $k=12$ as an appropriate constant for any level of the fractal structure $\ef$). 
Thus, $\bc(K)=\tres(K)=2$.
\end{proof}

Recall that the next corollary follows as a consequence of both Theorem \ref{teo:dim3} and classical Moran's one.

\begin{cor}(\cite[Corollary 4.22]{DIM3})\label{cor:dim3}
Let $\F$ be an Euclidean IFS satisfying the OSC, whose associated IFS-attractor is $K$. Thus, if $\ef$ is the natural fractal structure on $K$ as a self-similar set, then the following chain of equalities holds:
$$\bc(K)=\tres(K)=\h(K).$$
\end{cor}

However, as the next counterexample states, Corollary \ref{cor:dim3} cannot be improved in the sense that the OSC is still needed in order to achieve such a result.

\begin{ct}(\cite[Remark 4.23]{DIM3})\label{ct:12}
There exists an Euclidean IFS $\F$ (not satisfying the OSC), whose associated IFS-attractor $K$, equipped with its natural fractal structure $\ef$ as a self-similar set, satisfies that $\h(K)\neq \tres(K)$.
\end{ct}

\begin{proof}
Let $I=\{1,2,3\}$ be a finite index set, $\F=\{f_i:i\in I\}$, and $(\mathbb{R},\F)$ be an Euclidean IFS, whose associated IFS-attractor $K=[0,1]$, satisfies the following Hutchinson's equation: $K=\bigcup_{f\in \F}f(K)$. Further, let
$f_i:\mathbb{R}\lto \mathbb{R}$ be the contractions defined as
\begin{equation*}
f_{i}(x)=\left\{
\begin{array}{ll}
\hbox{$\frac{x}{2}$} & \hbox{if $i=1$;} \\
\hbox{$\frac{x+1}{2}$} & \hbox{if $i=2$;} \\
\hbox{$\frac{2x+1}{4}$} & \hbox{if $i=3$.}
\end{array}
\right.
\end{equation*}
Moreover, let also $\ef$ be the natural fractal structure on $K$ as a self-similar set. Firstly, observe that all the contractions $f_i$ are similarities having a common similarity factor, equal to $1/2$. Therefore, $K$ is a strict self-similar set on the real line. Thus, Theorem \ref{teo:dim3} gives that $\tres(K)$ is the solution of $\sum_{i\in I}c_i^s=1$. Thus, $\tres(K)=\log 3/\log 2$. On the other hand, both \cite[Theorem 4.10]{DIM3} and \cite[Theorem 4.15]{DIM1}, lead to $\uno(K)=\dos(K)=\log 3/\log 2=\tres(K)$, since all the elements in level $n$ of $\ef$ have a diameter whose order is equal to $1/2^n$. Finally, we affirm that $\F$ does not satisfy the OSC. In fact, suppose the opposite. Then Corollary \ref{cor:dim3} would imply that $\tres(K)$ must be equal to $\h(K)=1$, which is a contradiction. Hence, the result follows.
\end{proof}

\section{Counterexamples to highlight that some fractal dimensions do not coincide in general}\label{sec:agree}

Fractal structures allow to generalize the classical models of fractal dimension in the more general context of GF-spaces. To deal with, some new models for fractal dimension with respect to a fractal structure have been explored along this paper (recall Subsection \ref{sub:models}). It turns out that each one of them presents some properties and features that make it more appropriate than the rest of them depending on each specific context. Moreover, some results have been explored in previous works in order to achieve equalities between those fractal dimension definitions.
We would like to point out that these kind of theoretical results allow to interplay these definitions to powerful their effect, as Falconer states in the context of classical fractal dimension models (see \cite[Subsection 3.2]{FAL90}).
It is worth mentioning that this kind of theoretical results involve conditions regarding the elements of a fractal structure and restrictions about the construction of IFS-attractors, as well. In this section, though, we provide some counterexamples that emphasize the fact that several fractal dimension models for a fractal structure do not agree, in general.

To start with, we state that fractal dimensions I and II do not always match.

\begin{ct}(\cite[Remark 3.11 \& Remark 4.9]{DIM1})\label{ct:5}
There exist a subset $F$ of a given space $X$, as well as a fractal structure $\ef$ on $X$, such that $\uno(F)\neq \dos(F)$. 
\end{ct}

\begin{proof}
Let $X=\R$, $C$ be the middle third Cantor set, and $\ef$ be the natural fractal structure on $C$ as a self-similar set. Then it holds that
$$\dim_{\ef}^{2}(C)
=\lim_{n\to \infty} \frac{\log 2^{n}}{-\log 3^{-n}}
=\frac{\log 2}{\log 3},$$
since in each level $\Gamma_{n}$ of that fractal structure, there are $2^{n}$ ``subintervals'' whose diameters are equal to $1/3^n$. That argument also remains valid to get $\uno(C)=1$.
%Let $C$ be the classical middle third Cantor set, and let $\ef_{1}$ be the natural fractal structure on $C$ as a self-similar set (recall Definition 2.4). Then an easy calculation leads to $\dim_{\ef_{1}}(C)=1$, since on each level $\Gamma_{1,n}$ of the fractal structure $\ef_1$, there are $2^n$ ``subintervals'' whose length is equal to $1/3^n$.
%Let $\ef_{1}$ be the natural fractal structure on $C_{1}$ as a self-similar set, where $C_{1}$ is the middle third Cantor set on $[0,1]$. Let also $\ef_{2}$ be a fractal structure on $C_{2}=[2,3]$
%given by $\ef_{2}=\{\Gamma_{2,n}:n\in \mathbb{N}\}$, with $\Gamma_{2,n}=\{[\frac{k}{2^{2n}},\frac{k+1}{2^{2n}}]:k\in \{2^{2n+1},2^{2n+1}+1,\ldots, 3\cdot 2^{2n}-1\}\}$ for all natural number $n$.
%Consider also $\ef=\{\Gamma_{n}:n\in \mathbb{N}\}$ as a fractal structure on $C=C_{1}\cup C_{2}$, where $\Gamma_{n}=\Gamma_{1,n}\cup \Gamma_{2,n}$ for all $n\in \mathbb{N}$. A simple calculation leads to $\dim_{\ef}^{2}(C_{1})=\log 2/ \log 3$ and $\dim_{\ef}^{2}(C_{2})=1$, while $\dim_{\ef}^{2}(C)=\log 4/ \log 3> 1$.
\end{proof}

The following counterexample points out that fractal dimension III can be different from both fractal dimensions I \& II.

\begin{ct}\label{ct:11}
There exists an Euclidean IFS $\F$, whose associated IFS-attractor $K$, equipped with its natural fractal structure $\ef$ as a self-similar set, satisfies that $$\uno(K)=\dos(K)\neq \tres(K).$$
\end{ct}

\begin{proof}
In fact, let $I=\{1,2,3\}$ be a finite index set, $\F=\{f_i:i\in I\}$, and $(\mathbb{R},\F)$ be an Euclidean IFS, whose associated IFS-attractor $K=[0,1]$, satisfies the following Hutchinson's equation: $K=\bigcup_{f\in \F}f(K)$. Moreover, let 
$f_i:\mathbb{R}\lto\mathbb{R}$ be the contractions given by
\begin{equation*}
f_{i}(x)=\left\{
\begin{array}{ll}
\hbox{$\frac{x}{2}$} & \hbox{if $i=1$;} \\
\hbox{$\frac{x+2}{4}$} & \hbox{if $i=2$;} \\
\hbox{$\frac{x+3}{4}$} & \hbox{if $i=3$.}
\end{array}
\right.
\end{equation*}
On the other hand, let $\ef$ be the natural fractal structure on $K$ as a self-similar set. It becomes also clear that the maps $f_i$ are similarities, so the IFS-attractor $K$ is a strict self-similar set. 
We also affirm that $\F$ satisfies the OSC. In fact, just take $V=(0,1)$ as an appropriate open set. Thus, notice that
\cite[Corollary 4.22]{DIM3} leads to $\tres(K)=\bc(K)=\h(K)=1$. Furthermore, note that in each level $n$ of the fractal structure $\ef$, there are $3^n$ subintervals contained in $[0,1]$. Hence, since $\delta(K,\Gamma_n)=1/2^n$, then $\dos(K)=\log 3/\log 2=\uno(K)$. In fact, that result becomes straightforward by applying both Definition \ref{def:typebox} (\ref{def:dim2}) and \cite[Theorem 4.15]{DIM1}, respectively.
\end{proof}

It is worth mentioning that the unique choice in order to study the fractal dimension of a curve through the classical models for fractal dimension is to calculate the fractal dimension for the graph of the curve. In addition to that, fractal dimension III (introduced in Definition \ref{def:dim3} and calculated as in \cite[Theorem 4.7]{DIM3}) still could be applied for that purpose. In fact, this fractal dimension model for a fractal structure considers the structure of the curve as well as the complexity of the procedure applied in order to generate it. To deal with, first we define an appropriate \emph{induced} fractal structure on (the parametrization of) any curve from a fractal structure on the closed unit interval.

\begin{defn}(\cite[Definition 3.1]{DIMC})\label{def:fsind}
Let $\alpha:[0,1]\longrightarrow X$ be a parametrization of a curve, $X$ be a metric space, and $\ef$ be a fractal structure on $[0,1]$. Thus, the fractal structure induced by $\ef$ on the image set\, $\alpha([0,1])\subseteq X$, is defined as the countable family of coverings $\gf=\{\Delta_n\}_{n\in \mathbb{N}}$, where its levels are given by $\Delta_n=\alpha(\Gamma_n)=\{\alpha(A):A\in \Gamma_n\}$. 
\end{defn}

The fractal dimension for the parametrization of a given curve can be defined from both the induced fractal structure (as given in Definition \ref{def:fsind}) and the fractal dimension III model, as follows:

\begin{defn}(\cite[Definition 3.3]{DIMC})\label{def:dcur}
Let $\rho$ be a distance on $X$, and $\alpha:[0,1]\lto X$ be a parametrization of a curve. Let also $\ef$ be a fractal structure on $[0,1]$, and $\gf$ be the fractal structure induced by $\ef$ on the image set $\alpha([0,1])\subseteq X$. Then the fractal dimension of the (parametrization of the) curve $\alpha$ is defined as $\dim_{\ef}(\alpha)=\tresb(\alpha([0,1]))$. Further, if no additional information regarding the starting fractal structure $\ef$ is provided, then we will assume that $\ef$ is the natural fractal structure on $[0,1]$. In that case, the fractal dimension of the curve will be denoted, merely, by $\dc=\tresb(\alpha([0,1]))$.
\end{defn}

The next counterexample we provide presents a curve which, as the classical Hilbert's curve does, also fills the whole unit square,
though its fractal dimension does not agree with its box dimension nor its
Hausdorff dimension. Thus, it shows that Definition \ref{def:dcur} of fractal
dimension results more accurate than the classical
models of fractal dimension, since it also takes into account the procedure followed to construct it.

%\begin{ejem}[A modified Hilbert's curve]\label{ejem:dhil2}
%Let us consider a modified Hilbert's curve $\beta$ which crosses twice
%some elements of each level of the induced fractal structure $\gf$ and let us calculate
%its fractal dimension. Thus, we obtain that its fractal dimension does not agree
%with its box-counting dimension nor its Hausdorff dimensions (calculated for the image set $\beta(I)$).
%\end{ejem}

\begin{ct}\label{ct:13}
A modified Hilbert's curve $\beta$, which crosses twice
some elements in each level of its induced fractal structure $\gf$, can be constructed in order to prove that its fractal dimension is different from both
its box and its Hausdorff dimensions, which are calculated with respect to the image set $\beta([0,1])$.
\end{ct}

\begin{proof}
Indeed, let $\ef=\{\Gamma_{n}\}_{n\in \mathbb{N}}$ be a fractal structure on $[0,1]$, whose levels are given by
$$\Gamma_{n}=\Big\{\Big[\frac{k}{5^{n}},\frac{k+1}{5^{n}}\Big]:k\in \{0,1,\ldots,5^{n}-1\}\Big\}.$$
Let us consider also the curve $\beta:[0,1]\lto Y$, that we will define through \cite[Theorem 3.6]{DIMC}, where $Y=[0,1]\times [0,1]$ has been equipped
with the Euclidean distance. Moreover, let also $\gf$ be the fractal structure induced by $\ef$ on $\beta([0,1])$.
Thus, the definition of the modified Hilbert's curve is given by the sequence of maps $\{\beta_n\}_{n\in \mathbb{N}}$, where the description of each map
$\beta_{n}:\Gamma_{n}\lto \Delta_{n}$, is illustrated
in Figure \ref{fig:hilbert2} for its first two levels. Observe that the polygonal line shows the
procedure that must be followed in order to fill the whole unit square in each stage.\newline
Moreover, notice that the elements in the first level $\Delta_1$ of the induced fractal
structure are given as
$\beta([0,\frac{1}{5}])=\beta([\frac{1}{5},\frac{2}{5}])=[0,\frac{1}{2}]^{2}$,
$\beta([\frac{2}{5},\frac{3}{5}])=[\frac{1}{2},1]\times [0,\frac{1}{2}]$,
$\beta([\frac{3}{5},\frac{4}{5}])=[\frac{1}{2},1]^{2}$, and
$\beta([\frac{4}{5},1])=[0,\frac{1}{2}]\times [\frac{1}{2},1]$, for instance. Notice that the subsequent levels
can be obtained in a similar way.
In addition to that, it holds that $\dim_{B}(\beta([0,1]))=\dim_{H}(\beta([0,1]))=2$, since the modified
Hilbert's curve fills the whole unit square.\newline
On the other hand, $\mathcal{H}_{n}^{s}(\beta([0,1]))=(\sqrt{2})^{s}\,
(\frac{5}{2^{s}})^{n}$, since there are $5^{n}$ ``subsquares'' (having a diameter equal to
$\sqrt{2}/2^{n}$, each) which intersect $\beta([0,1])$ in level $n$ of the induced
fractal structure $\gf$. Accordingly,
\begin{equation*}
\mathcal{H}^{s}(\beta([0,1]))=\left\{
\begin{array}{ll}
\hbox{$\infty$} & \hbox{if $s<\frac{\log 5}{\log 2}$;}\\
\hbox{$\sqrt{5}$} & \hbox{if $s=\frac{\log 5}{\log 2}$;} \\
\hbox{$0$} & \hbox{if $s>\frac{\log 5}{\log 2},$} 
\end{array}
\right.
\end{equation*}
which leads to $\dim(\beta)=\log 5/ \log 2$. 
\begin{center}
\begin{figure}
\centering
\begin{tabular}{cc}
\includegraphics[width=60mm, height=53mm]{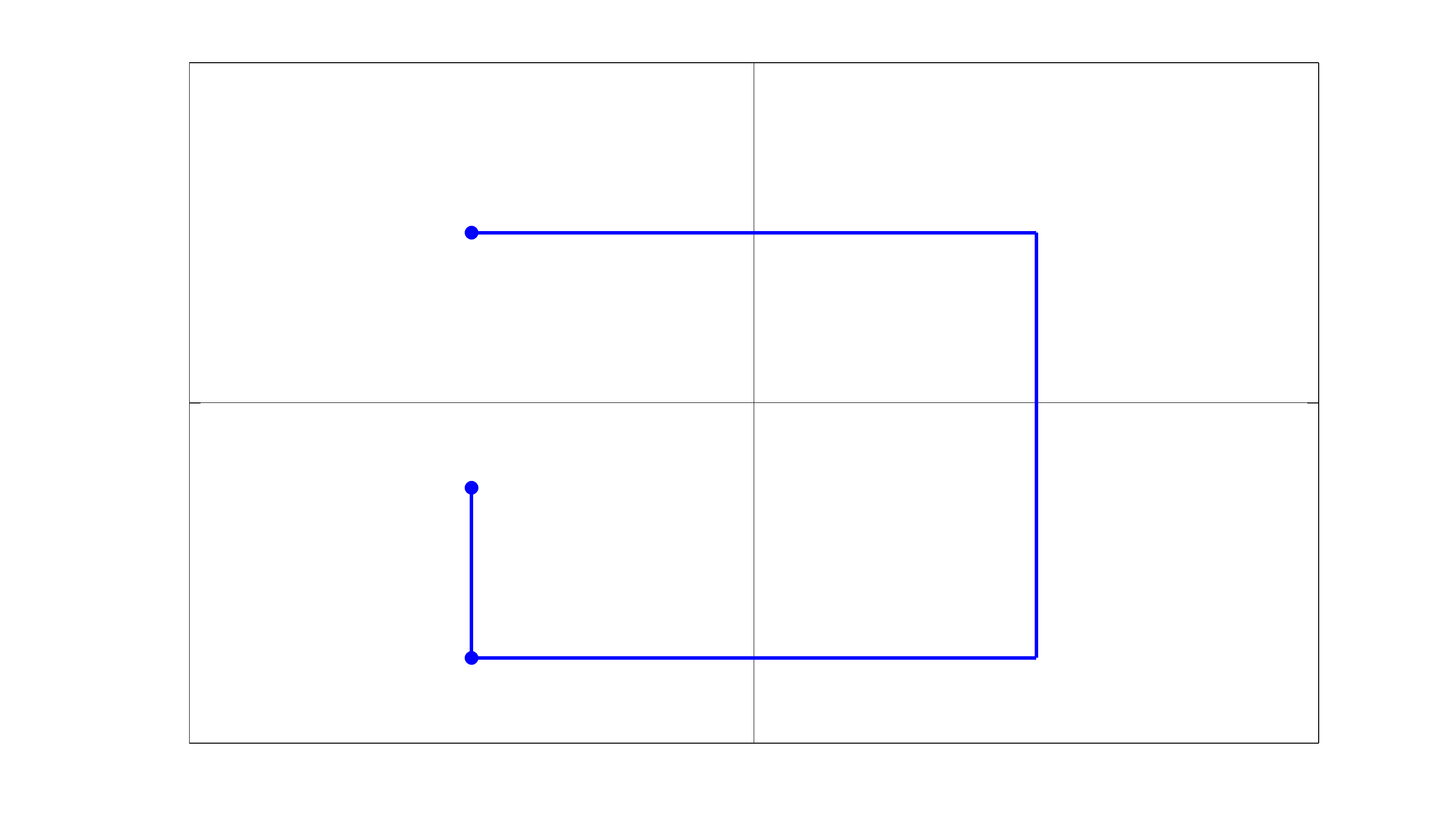} &
\includegraphics[width=60mm, height=53mm]{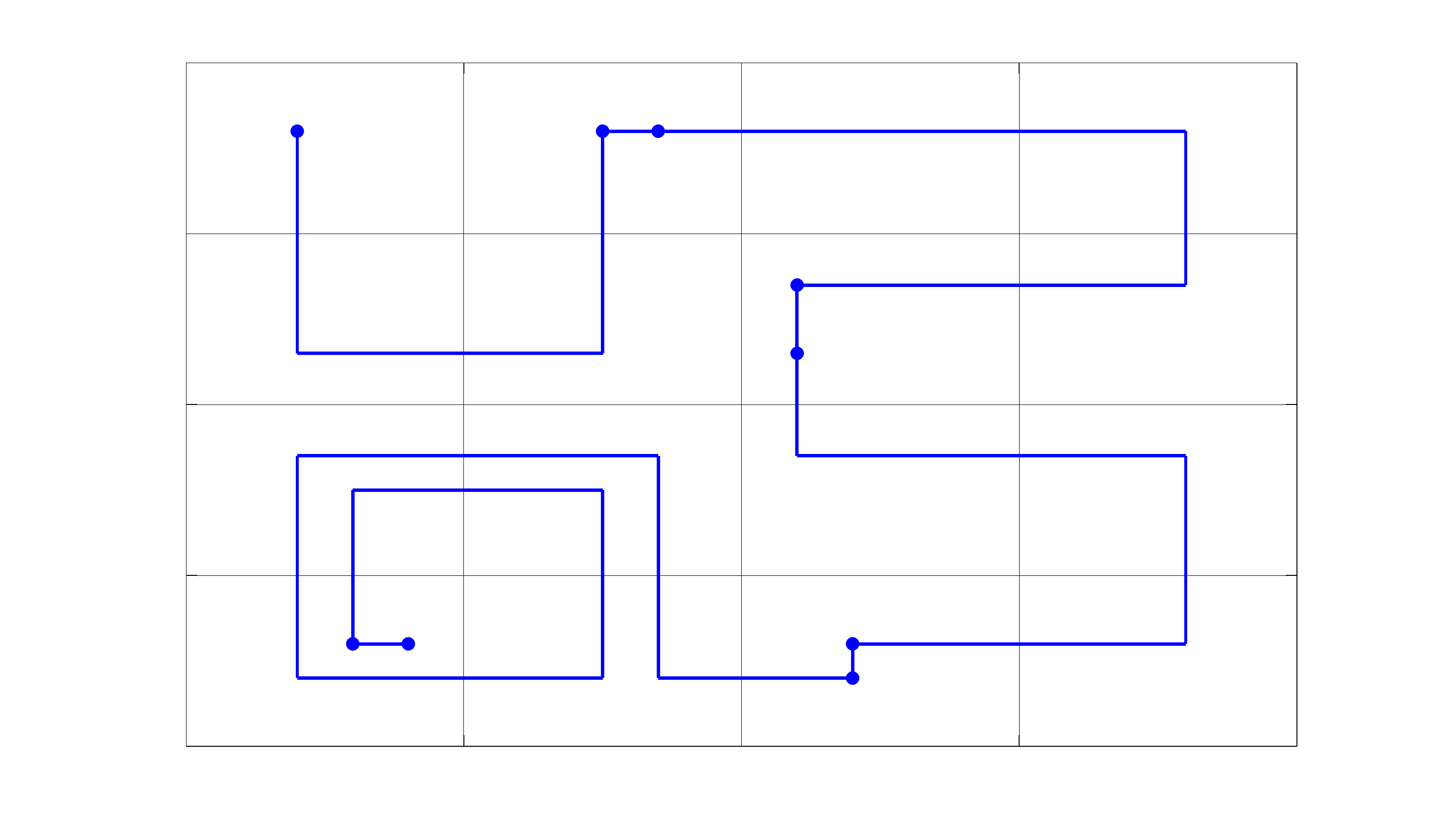}
\end{tabular}
\caption{First two levels in the construction of the modified Hilbert's curve.}\label{fig:hilbert2}
\end{figure}
\end{center}
\end{proof}

\begin{ct}
%There exist a fractal structure $\ef$ on a space $X$ and a subset $F \subseteq X$ such that $\tres(F) \not=\cinco(F)$, and $\tres(F) \not=\cuatro(F)$.
There exist a certain subspace $X$, as well as a fractal structure $\ef$ on $X$, such that $\tres(X) \not=\cinco(X)$, and $\tres(X) \not=\cuatro(X)$.
\end{ct}

\begin{proof}
Let $X=\beta([0,1])$, and $\ef$ be as $\gf$ in Counterexample \ref{ct:13}. Firstly, recall that Counterexample \ref{ct:13} gives that $\tres(X)=\log 5/ \log 2$. On the other hand,
note that the fractal dimension V of $X$ with respect to the natural fractal structure is equal to $2$, and hence, $\cinco(X)=2$, due to the relationship between $\ef$ and the natural fractal structure.
Similarly, $\cuatro(X)=2$.
\end{proof}

Interestingly, the fractal dimension IV introduced in Definition \ref{def:456} (\ref{eq:hnk}) provides an intermediate model between the Hausdorff and the box dimensions, as the following counterexample establishes.

\begin{ct}\label{ct:12}
\begin{enumerate}[(1)]
%There exist a subset $F$ of a metric space $(X,\rho)$, and a fractal structure $\ef$ on $X$, such that
\item There exist a subset $F$ of a metric space $(X,\rho)$, and a fractal structure $\ef$ on $X$, such that $\cuatro(F)<\bc(F)$.
\item There exist a subset $F$ of a metric space $(X,\rho)$, and a fractal structure $\ef$ on $X$, such that $\dih(F)<\cuatro(F)$.
\end{enumerate}
\end{ct}

\begin{proof}
\begin{enumerate}[(1)]
\item Let $\ef$ be the natural fractal structure on $[0,1]$, and $F=\{0,1,\frac{1}{2},\frac{1}{3},\ldots\}$. First, note that \cite[Example 3.5]{FAL90} yields $\bc(F)=1/2$. Thus, since $F$ is a compact Euclidean subset, then \cite[Theorem 3.12]{DIM4} leads to $\cuatro(F)=\dih(F)=0$, since $F$ is countable.
\item Let $\ef$ be the natural fractal structure on $X=[0,1]$, and $F=\mathbb{Q}\cap X$. Thus, note that \cite[Theorem 3.13]{DIM4} yields $\cuatro(F)=\dih(\overline{F})=1$, though $\dih(F)=0$.
\end{enumerate}
\end{proof}

Further, it holds that fractal dimensions IV and V do not always coincide.

\begin{ct}
%There exists a fractal structure $\ef$ on a space $X$ and a subset $F \subseteq X$ such that $\cuatro(F) \not=\cinco(F)$.
There exist a subset $F$ of a certain space $X$, as well as a fractal structure $\ef$ on $X$, such that $\cuatro(F) \not=\cinco(F)$.
\end{ct}

\begin{proof}
Let $\ef$ be the natural fractal structure on $X=[0,1]$, and $F=\mathbb{Q}\cap X$. Thus, \cite[Theorem 3.13]{DIM4} leads to $\cuatro(F)=\dih(\overline{F})=1$, whereas \cite[Theorem 3.10]{DIM4} gives that $\cinco(F)=\dih(F)=0$.
\end{proof}

Nevertheless, though both Hausdorff dimension and fractal dimension VI implicitly assume that $\delta(F,\Gamma_n)\to 0$, they do not coincide, in general, as the following example establishes.

\begin{ct}
%There exist a fractal structure $\ef$ on a space $X$ and a subset $F \subseteq X$ such that $\tres(F) \not=\cinco(F)$, and $\tres(F) \not=\cuatro(F)$.
There exist a certain space $X$, as well as a fractal structure $\ef$ on $X$, such that $\h(X)\neq \seis(X)$.
\end{ct}

\begin{proof}
Let $\ef=\{\Gamma_n\}_{n\in \N}$ a fractal structure on $X=\R^2$, whose levels are given as follows:
\begin{equation*}
\Gamma_n=\Bigg\{\Bigg[\frac{k}{2^n},\frac{k+1}{2^n}\Bigg]\times \{x\}:k\in \Z,x\in \R\Bigg\}.
\end{equation*}
In other words, it is considered the natural fractal structure on each horizontal straight line in $\R^2$. Thus, while the fractal dimension VI for each horizontal subinterval is equal to $1$, the fractal dimension VI for each vertical subinterval is equal to $\infty$, since $\mathcal{A}_{\delta,6}(X)=\emptyset$. This is due to the fact that there are no countable coverings by elements in the fractal structure $\ef$, since its elements are horizontal segments which do intersect the vertical ones just in a single point. It is also worth mentioning that the topology induced by such a fractal structure $\ef$ does not agree with the usual topology.
\end{proof}

Finally, we would like to end the paper through an interesting open question involving fractal dimensions V and VI.

\begin{prob}
Does there exist a fractal structure $\ef$ on a space $X$ and a subset $F \subseteq X$ with $\delta(F,\Gamma_n)\to 0$ and such that $\cinco(F)\neq\seis(F)$?
\end{prob}

Recall that some conditions regarding the fractal structure were provided in \cite[Theorem 3.7]{DIM4} in order to reach the equality $\cinco(F)=\seis(F)$.

\end{document}